\documentclass[11pt]{article}
\usepackage{verbatim}
\usepackage[margin=1in]{geometry}
\usepackage{amsmath,amssymb,amsthm,mathtools}
\usepackage{enumitem}
\usepackage{hyperref}
\usepackage[colorinlistoftodos,textsize=scriptsize]{todonotes}
\usepackage{authblk}

\newtheorem{theorem}{Theorem}[section]
\newtheorem{lemma}[theorem]{Lemma}
\newtheorem{proposition}[theorem]{Proposition}

\theoremstyle{definition}

\DeclareMathOperator{\lcm}{lcm}
\DeclareMathOperator{\CRT}{CRT}

\newcommand{\Z}{\mathbb Z}

\newcommand{\seqnum}[1]{\href{https://oeis.org/#1}{\rm \underline{#1}}}

\title{Asymptotic Behavior of Iterated Sets of Remainders}

\author{
  \begin{minipage}{0.45\textwidth}
    \centering
    \textbf{Omkar Baraskar} \\
    {\small Cheriton School of Computer Science\\ University of Waterloo \\ \textit{ON N2L 3G1, Canada} \\
    \texttt{obaraska26@gmail.com}}
  \end{minipage}
  \hfill
  \begin{minipage}{0.45\textwidth}
    \centering
    \textbf{Prashant Gokhale} \\
    {\small University of Wisconsin-Madison \\
    \texttt{prashant.gokhale@wisc.edu}}
  \end{minipage}
  
  \vspace{0.8cm} 
  
  \begin{minipage}{0.45\textwidth}
    \centering
    \textbf{Sarvagya Jain} \\
    {\small Department of Mathematics and Statistics \\
    University of Turku \\
    \textit{20014 Turku, Finland} \\
    \texttt{sajain@utu.fi}}
  \end{minipage}
  \hfill
  \begin{minipage}{0.45\textwidth}
    \centering
    \textbf{Adam Kie\.zun} \\
    {\small Meta Superintelligence Labs} \\ 
    \texttt{akiezun@meta.com} 
  \end{minipage}
}

\date{}

\newcommand{\comP}[1]{%
  \relax
}

\newcommand{\comS}[1]{%
 \relax
}

\newcommand{\comA}[1]{%
  \relax
}

\newcommand{\comO}[1]{%
  \relax
}

\begin{document}
\maketitle
\begin{abstract}
For a positive integer \(n\), let
\[
S_0(n)=\{1,2,\ldots,\lfloor n/2\rfloor\},\qquad
S_{j+1}(n)=\{n\bmod k:k\in S_j(n)\setminus\{0\}\},
\]
and put \(s_j(n) \coloneq |S_j(n)|\). The sets defined above arise naturally in the study of the length of the Pierce series expansion of a rational number. 
In \cite{Ba-Vu}, Baraskar and Vukusic conjectured that for every fixed \(j\geq 2\), the limit \[\lim_{n\to\infty} s_j(n)/n\] does not exist. In this paper, we prove this conjecture in the affirmative. 

Moreover, we define a new class of iterated remainder sets \(T_j(n)\) that naturally arises from the study of the length of the Engel series expansion of a rational number. We analogously study the asymptotic behavior of \(t_j(n) \coloneq |T_j(n)| \). We show that
\[
\lim_{n\to\infty}\frac{t_j(n)}n
\]
exists precisely for \(j\in\{0,1\}\) and fails to exist for every fixed
integer \(j\geq2\).

\end{abstract}
\section{Introduction}
\comS{Go over the draft of the paper and find the errors and gaps in it, and suggest how to fix them. Classify the errors into three categories:
1. Typos, 2. Grammar mistakes and 3. Mathematical errors, 4. Unclear conventions, 5. insufficient exposition. Also, find displays that have been numbered but never referred to later. }

In \cite{Pierce1929} Pierce showed that every real number $x \in (0,1]$ has a unique expansion as 

$$x = \frac{1}{x_1}-\frac{1}{x_1x_2}+\frac{1}{x_1x_2x_3}-\dots,$$
where $(x_i)$ form a strictly increasing sequence\footnote{When the expansion has at least two terms and terminates with the last term $\frac{(-1)^{n+1}}{x_1\cdots x_n}$, we additionally enforce $x_n > x_{n-1} + 1$. This is done to ensure uniqueness as $\frac{1}{k} = \frac{1}{k-1} - \frac{1}{k(k-1)}$ for $k\geq 2$.} of positive integers. This expression was named Pierce expansion by Shallit in \cite{Sha1986}.


Similarly, in \cite{Engel1913} Engel showed that every positive real number $y$ has a unique expansion, called the Engel series
$$y = \frac{1}{y_1}+\frac{1}{y_1y_2}+\frac{1}{y_1y_2y_3}+\dots,$$
where $(y_i)$ form a nondecreasing sequence of positive integers. If the expansion is infinite, we enforce that the sequence $y_i$ is not eventually constant, in order to exclude alternative infinite expansions for rational numbers and ensure uniqueness.

\comO{Should we keep the fact that $y_i$'s are not eventually constant for inifinite Engel expansions in the introduction. Feels kind of random.}
\comS{For uniqueness this is needed, also if don't include this then one will have two reps for rational... one of which is infinite}

Both the Pierce and Engel expansions have finitely many terms if and only if the input is rational.

Erd\H{o}s and Shallit in~\cite[Observation~1]{Er-Sh} noted that
the algorithm generating the Pierce expansion of a rational number can be
described in terms of successive integer remainders. Fix a positive integer
\(n\) and \(1\leq a\leq n\). Set \(a_0:=a\), and, as long as \(a_j>0\),
define
\[
    q_{j+1}:=\left\lfloor\frac{n}{a_j}\right\rfloor,
    \qquad
    a_{j+1}:=n-q_{j+1}a_j=n\bmod a_j.
\]
Thus
\[
    n=q_{j+1}a_j+a_{j+1},
    \qquad
    0\leq a_{j+1}<a_j.
\]
We define \(\mathcal P(n,a)\) to be the least integer \(j\geq1\) for
which \(a_j=0\). Erd\H{o}s and Shallit observed that \(\mathcal P(n,a)\) is the length of the Pierce expansion of $a/n$. The result ~\cite[Theorem~2]{Er-Sh}, together with the algorithmic
description, motivate the study of the Pierce-type length function
\[
    \mathcal P(n):=
    \max_{1\leq a\leq n}\mathcal P(n,a).
\]

There is an analogous description for Engel expansions. Fix a positive
integer \(n\) and \(1\leq a\leq n\). Set \(a_0:=a\), and, as long as
\(a_j>0\), define
\[
    q_{j+1}:=\left\lceil\frac{n}{a_j}\right\rceil,
    \qquad
    a_{j+1}:=q_{j+1}a_j-n=(-n)\bmod a_j.
\]
Thus
\[
    q_{j+1}a_j=n+a_{j+1},
    \qquad
    0\leq a_{j+1}<a_j.
\]
We define \(\mathcal E(n,a)\) to be the least integer \(j\geq1\) for
which \(a_j=0\). Erd\H{o}s and Shallit in~\cite[p. 45]{Er-Sh}, observed that \(\mathcal E(n,a)\) is the length of the Engel expansion of $a/n$. The result~\cite[Theorem~5]{Er-Sh}, together with the analogous
algorithmic structure, motivate the Engel-type length function
\[
    \mathcal E(n):=
    \max_{1\leq a\leq n}\mathcal E(n,a).
\]

Building on the work of~\cite[Theorem 2]{Er-Sh}, Chase and Pandey in~\cite[Theorem 1.1, Theorem 1.2]{Ch-Pa} showed that
\begin{align}\label{eqn: bounds P(n) ch-pa}
    \frac{\log n}{\log\log n}\ll \mathcal{P}(n)\ll_\varepsilon n^{\frac{1}{3}-\frac{2}{177}+\varepsilon}.
\end{align}
For the Engel-type length function, ~\cite[Theorem 5]{Er-Sh} gives
\begin{align*}
    \mathcal{E}(n)\ll_\varepsilon n^{\frac{1}{3}+\varepsilon}.
\end{align*}

Let \(n\) be a positive integer. Baraskar and Vukusic in~\cite{Ba-Vu} studied the
remainder set
\[
S_1(n):=\{n\bmod k:1\leq k\leq \lfloor n/2\rfloor\}
\]
and denoted its cardinality by \(s_1(n)\). This sequence is listed as \seqnum{A283190} in the OEIS (On-Line Encyclopedia of Integer Sequences) \cite{oeis}. More generally, Baraskar and Vukusic
defined the iterated remainder sets
\[
S_0(n):=\{1,2,\ldots,\lfloor n/2\rfloor\},\qquad
S_{j+1}(n):=\{n\bmod k:k\in S_j(n)\setminus\{0\}\},
\]
and wrote
\[
s_j(n):=|S_j(n)|.
\]
They showed that these sets are closely related to the Pierce-type length function \(\mathcal{P}(n)\); 
the precise relation is stated below.
\begin{lemma}\label{lemma: pierce length criterion}
Let \(n\geq3\) and \(t\geq2\). Then $\mathcal{P}(n) = t$ if and only if  \(S_{t-1}(n)=\{0\}\).
\end{lemma}
\begin{proof}
    See \cite[Lemma 19]{Ba-Vu}. We note that there is minor indexing issue in the original argument. 
    
\end{proof}

On the Stack Exchange website \cite{Israel2017}, Robert Israel numerically observed that the value of $s_1(n)/n$ seems to converge to $0.2296$ as $n \to \infty$ and asked whether the limit exists and, if so, what its value is. The user Empy2 showed that the limit exists and is equal to 

\begin{equation}\label{eq:asymp}
    \lim_{n \to \infty} \frac{s_1(n)}{n}
    = \sum_{p \text{ prime}} \frac{1}{p(p+1)} 
        \cdot \prod_{\substack{q \text{ prime} \\ q < p}} \left( 1 - \frac{1}{q} \right)
    \approx  0.2296. 
\end{equation}

In~\cite{Ba-Vu}, Baraskar and Vukusic studied similar questions for $s_j(n)$ and  showed that for each fixed \(j\geq 0\), the quantity
\(s_j(n)\) grows linearly along suitable ranges of \(n\). In particular,
they proved the bounds
\[
\frac{1}{(j+2)!}
\leq
\liminf_{n\to\infty}\frac{s_j(n)}{n}
\leq
\limsup_{n\to\infty}\frac{s_j(n)}{n}
\leq
\frac{1}{j+2}.
\]

Based on numerical evidence, Baraskar and Vukusic conjectured that $s_j(n)/n$ does not converge for any fixed $j\ge2$; see~\cite[Problem 2]{Ba-Vu}.

Our first result resolves the above conjecture in the affirmative.
\begin{theorem}\label{thm:non-conv-s}
For \(j\geq 2\), define
\[
\Delta_j^P
:=
\limsup_{n\to\infty}\frac{s_j(n)}{n}
-
\liminf_{n\to\infty}\frac{s_j(n)}{n}.
\]
Then
\[
\Delta_j^P\geq \frac{j-1}{(j+4)!}.
\]
In particular, for every fixed integer \(j\geq 2\), the limit
\[
\lim_{n\to\infty}\frac{s_j(n)}{n}
\]
does not exist.

Moreover,
\[
\frac{1}{(j+2)!}
\leq
\liminf_{n\to\infty}\frac{s_j(n)}{n}
\leq
\limsup_{n\to\infty}\frac{s_j(n)}{n}
\leq
\exp\left(-\frac{j}{e}+O(\log j)\right),
\]
where the implied constant is absolute.
\end{theorem}



Baraskar and Vukusic showed~\eqref{eq:asymp} in a more quantitative form by showing that 
\[s_1(n) = n\sum_{p \text{ prime}} \frac{1}{p(p+1)} 
        \cdot \prod_{\substack{q \text{ prime} \\ q < p}} \left( 1 - \frac{1}{q} \right) + O\left(\frac{n}{\log n\log\log n}\right).\]
In~\cite[Problem~1]{Ba-Vu}, they asked for an improvement of the error term and, in the discussion preceding the problem, they made the
conjectural suggestion that the correct order of magnitude should be \(O(n^{1/3})\). In forthcoming work, we show that this conjecture is false. More precisely, for every \(\varepsilon>0\), we prove the upper bound
\[
O_\varepsilon\left(
n\exp\left(-(\sqrt{2}-\varepsilon)
\sqrt{\log n\log\log n}\right)
\right),
\]
while also showing that the error term is not
\[
O\left(
n\exp\left(-(\sqrt{2}+\varepsilon)
\sqrt{\log n\log\log n}\right)
\right)
\]
for any \(\varepsilon>0\). In particular, no power-saving error term is possible.

One may also wish to study analogues of the above sets for Engel expansions. To this end, we introduce the following iterated remainder sets
\begin{align*}
    T_0(n):=\{1, 2, \dots, n-1\},\quad T_{j+1}(n) := \{(-n)\bmod k: k\in T_{j}(n)\setminus \{0\}\},
  \end{align*}
and write
\begin{align*}
    t_j(n):=|T_j(n)|.
\end{align*}
We relate the sets $T_j(n)$ to $\mathcal{E}(n)$ in the following lemma. 
\begin{lemma}\label{lemma: engel length criterion}
Let \(n\geq 2\) and \(\ell\geq 1\). Then $\mathcal{E}(n)= \ell$ if and only if \(T_\ell(n)=\{0\}\). 
\end{lemma}
We defer the proof of Lemma~\ref{lemma: engel length criterion} to Section~\ref{sec: Length Functions and Iterated Remainder Sets}.

In analogy with the study of $s_j(n)$, we show the following result for $t_j(n)$.
\begin{theorem}\label{thm:non-conv-t}
For \(j\geq 2\), define
\[
\Delta_j^E
:=
\limsup_{n\to\infty}\frac{t_j(n)}{n}
-
\liminf_{n\to\infty}\frac{t_j(n)}{n}.
\]
Then
\[
\Delta_j^E\geq \frac{1}{24\cdot 3^j}.
\]
In particular, for every fixed integer \(j\geq 2\), the limit
\[
\lim_{n\to\infty}\frac{t_j(n)}{n}
\]
does not exist.

Moreover,
\[
\frac{3}{2^{j+1}}-\frac{1}{2\cdot 3^j}
\leq
\liminf_{n\to\infty}\frac{t_j(n)}{n}
\leq
\limsup_{n\to\infty}\frac{t_j(n)}{n}
\leq
2^{-(j-1)/2}.
\]

When \(j=1\), the limit does exist, and
\[
\lim_{n\to\infty}\frac{t_1(n)}{n}
=
\sum_{t=2}^{\infty}\frac{1}{t(t-1)}
\left(
1-\prod_{\substack{p\leq t\\ p\ \mathrm{prime}}}
\left(1-\frac{1}{p}\right)
\right)
=
\sum_{p\ \mathrm{prime}}
\frac{1}{p(p-1)}
\prod_{\substack{q<p\\ q\ \mathrm{prime}}}
\left(1-\frac{1}{q}\right).
\]
\end{theorem}

The expression in the case \(j=1\) is notably similar to its Pierce analogue in~\eqref{eq:asymp}.

As an application of our analysis, we obtain an alternative proof of the
uniform lower bound for \(\mathcal P(n)\) established in
\cite[Theorem~1.2]{Ch-Pa}, together with an explicit lower bound for
\(\mathcal E(n)\). We record these bounds in the following theorem.

\begin{theorem}\label{thm:lower-bounds-length-functions}
The following statements hold.
\begin{enumerate}
    \item Define
    \[
    K(n):=\max\{k\geq 1:k!<n\}.
    \]
    For all sufficiently large \(n\),
    \[
    \mathcal{P}(n)\geq K(n)-1\geq (1-o(1))\frac{\log n}{\log\log n}.
    \]

    \item  For all sufficiently large \(n\),
    \[
    \mathcal{E}(n)\geq \left\lfloor \frac{\log(n-1)}{\log 2}\right\rfloor\geq \frac{\log n}{\log 2}-2.
    \]
\end{enumerate}
\end{theorem}
\subsection{Overview}
We briefly describe the organization of the paper and the principal ideas
underlying the proofs. We focus primarily on the Pierce setting, since the
Engel setting is treated by a parallel argument.

The first step is to decompose the iterated remainder set according to the
size of its elements. For \(t\geq 0\), write
\[
I_t(n):=
\left(\frac{n}{t+2},\frac{n}{t+1}\right]\cap\mathbb Z.
\]
For every fixed \(j\geq 1\), the positive elements of \(S_j(n)\) admit the
disjoint decomposition
\[
S_j(n)\setminus\{0\}
=
\coprod_{t\geq j+1}\bigl(S_j(n)\cap I_t(n)\bigr).
\]
We therefore study the contribution of each interval \(I_t(n)\) separately.

For \(j\geq 2\) and \(t\geq j+1\), define
\[
\mathcal D_P(j,t)
:=
\left\{
(p;q_2,\ldots,q_j):
p \text{ is prime and }
p<q_2<\cdots<q_j\leq t
\right\}.
\]
Repeatedly unwinding the remainder relations
\[
n=q_k r_{k-1}+r_k,
\qquad 2\leq k\leq j,
\]
shows that a trajectory terminating at an element of
\(S_j(n)\cap I_t(n)\) is encoded by a tuple
\(\mathbf d\in\mathcal D_P(j,t)\). In Section~3,
Lemma~\ref{lem:congruence-criterion pierce} makes this description
precise by showing that apart from a possible endpoint, an integer \(r\in I_t(n)\)
belongs to \(S_j(n)\) if and only if
\(
r\equiv A_P(\mathbf{d})n\pmod{m(\mathbf d)}
\)
for some \(\mathbf{d}\in\mathcal D_P(j,t)\), where \(A_P(\mathbf d)\) and \(m(\mathbf d)\) are
explicitly defined functions of the tuple \(\mathbf d\) in equation \eqref{eq:def-A}.

This congruence characterization reduces the local counting problem to
a finite problem in a cyclic group. Define
\[
M_P(j,t)
:=
\operatorname{lcm}
\left\{
m(\mathbf d):\mathbf d\in\mathcal{D}_P(j,t)
\right\}.
\]
For \(a\in\mathbb Z/M_P(j,t)\mathbb Z\), let
\[
U_{j,t}^{(P)}(a)
:=
\bigcup_{\mathbf d\in\mathcal{D}_P(j,t)}
\left\{
x\in\mathbb Z/M_P(j,t)\mathbb Z:
x\equiv A_P(\mathbf d)a\pmod{m(\mathbf d)}
\right\},
\]
and define its density by
\[
D_{j,t}^{(P)}(a)
:=
\frac{|U_{j,t}^{(P)}(a)|}{M_P(j,t)}.
\]
From this one concludes that (see Lemma~\ref{lem:interval-count pierce} below)
\[
\frac{\left|S_j(n)\cap I_t(n)\right|}{n}
=
\frac{
D_{j,t}^{(P)}
\bigl(n\bmod M_P(j,t)\bigr)
}{
(t+1)(t+2)
}
+
O_{j,t}\left(\frac{1}{n}\right).
\]
The main structural input is the cyclic dilation principle proved in
Lemma~\ref{lemma: cyclic dilation}. Applied to the union above, it shows
that
\[
D_{j,t}^{(P)}(0)
\leq
D_{j,t}^{(P)}(a)
\leq
D_{j,t}^{(P)}(u)
\]
for every \(a\in\mathbb Z/M_P(j,t)\mathbb Z\) and every unit
\(u\in(\mathbb Z/M_P(j,t)\mathbb Z)^\times\). Moreover, all units give
the same density. We therefore set
\[
\alpha_P(j,t):=D_{j,t}^{(P)}(0)
\qquad\text{and}\qquad
\beta_P(j,t):=D_{j,t}^{(P)}(1).
\]
The discussion above gives
\[
\liminf_{n\to\infty}\frac{s_j(n)}{n}
=
\sum_{t=j+1}^{\infty}
\frac{\alpha_P(j,t)}{(t+1)(t+2)}
\]
and
\[
\limsup_{n\to\infty}\frac{s_j(n)}{n}
=
\sum_{t=j+1}^{\infty}
\frac{\beta_P(j,t)}{(t+1)(t+2)}.
\]
It also allows one to show that the $\liminf$ and $\limsup$ are attained along the sequences \(n=m!\) and \(n=m!+1\), respectively. See Section~\ref{sec: The Cyclic Dilation Lemma} below for the details. 

In Section~\ref{sec: Characterizing the Convergence}, we establish the non-convergence of $\{s_j(n)/n\}_n$. It suffices to find one value of \(t\) for which
\(\beta_P(j,t)>\alpha_P(j,t)\), since the cyclic dilation lemma gives
\[
\beta_P(j,t)\geq\alpha_P(j,t)
\]
for every \(t\geq j+1\). We take \(t=j+2\). The tuples in $\mathcal{D}_P(j, j+2)$ give a particularly easy-to-study union of residue classes. For the dilation parameter \(0\), the corresponding residue classes are
subgroups of \(\mathbb Z/(j+2)!\mathbb Z\) and have substantial
overlap. For the dilation parameter \(1\), the shifted classes are essentially pairwise
disjoint. This gives
\[
\beta_P(j,j+2)-\alpha_P(j,j+2)
\geq
\frac{j-1}{(j+2)!};
\]
see Proposition~\ref{propn: limit non-existence pierce} below. It follows that
\[
\begin{aligned}
\limsup_{n\to\infty}\frac{s_j(n)}{n}
-
\liminf_{n\to\infty}\frac{s_j(n)}{n}
&=
\sum_{t=j+1}^{\infty}
\frac{\beta_P(j,t)-\alpha_P(j,t)}
     {(t+1)(t+2)}  \geq
\frac{j-1}{(j+4)!}>0.
\end{aligned}
\]
This proves that \(s_j(n)/n\) does not converge for any fixed
\(j\geq 2\). The remainder of the section is devoted to obtaining bounds on $\liminf$ and $\limsup$ as $j$ varies.

The Engel setting is treated by the same general strategy, with only minor modifications to the relevant intervals, quotient conditions, and congruence relations. For ease of comparison, the Engel arguments are presented alongside their Pierce counterparts throughout the paper. On a first reading, however, the reader may find it helpful to follow one setting in isolation before comparing the two parallel developments.

Finally, in Section~\ref{sec: Length Functions and Iterated Remainder Sets}, we first finish the proof of Lemma~\ref{lemma: engel length criterion}. Combining
the characterizations in Lemmas~\ref{lemma: pierce length criterion} and~\ref{lemma: engel length criterion} with the congruence criteria described in Section~\ref{sec: Congruence Characterization of the Remainder Sets} gives derivations of the lower bounds for \(\mathcal P(n)\) and \(\mathcal E(n)\).

\subsection{Notation}\label{subsec: notation}
All logarithms are natural logarithms. A hat over an entry of a tuple indicates that the
entry is omitted; for example,
\[
(3,4,\ldots,\widehat{a},\ldots,t)
\]
denotes the tuple obtained from \((3,4,\ldots,t)\) by deleting \(a\).

For integers \(a\) and \(M\geq1\), the expression \(a\bmod M\) denotes the
least nonnegative residue of \(a\) modulo \(M\), or the corresponding
element of \(\Z/M\Z\) when the ambient cyclic group is clear from the
context. 

For a function \(f\) and a nonnegative function \(g\), the notation
\[
f(x)\ll g(x)
\qquad\text{or, equivalently,}\qquad
f(x)=O(g(x))
\]
means that \(|f(x)|\leq Cg(x)\) for some constant \(C>0\) throughout the
range under consideration. Subscripts indicate the permitted dependence
of the implied constant; for example,
\[
f(x)\ll_{j,t}g(x)
\qquad\text{and}\qquad
f(x)=O_{j,t}(g(x))
\]
allow the constant to depend on \(j\) and \(t\), but on no other varying
parameter. For nonnegative functions $f$ and $g$, we write
\(
f(x)\asymp g(x)
\)
when both \(f(x)\ll g(x)\) and \(g(x)\ll f(x)\) hold. Furthermore,
\(
f(x)=o(g(x))
\)
means that \(f(x)/g(x)\to0\), and
\(
f(x)\sim g(x)
\)
means that \(f(x)/g(x)\to1\).

\subsection{Acknowledgments}
The authors would like to thank Ingrid Vukusic for many useful discussions and for pointing out the important reference~\cite{Simpson1986ExactCoverings}, specifically Lemma 2.3, which was used in an earlier draft of the proof and inspired the current version.

P. Gokhale thanks his advisors, Misha Khodak and Sandeep Silwal, for their encouragement, for fostering an environment that supports AI-assisted mathematical research, and for providing access to ChatGPT for Academic Researchers.

S. Jain would like to thank his advisor, Kaisa Matom{\"a}ki, for her guidance and encouragement and for providing him with access to a ChatGPT Business subscription. S. Jain was supported by the Research Council of Finland (grant numbers 346307, 364214, and 370133) and by the University of Turku Graduate School's Exactus fellowship while working on this project.

\subsection{AI Usage}
This preprint grew out of extensive collaboration between the authors and AI systems. The AI systems generated most of the main mathematical ideas going in the proofs, while the authors verified, refined, and wrote up the resulting arguments.

The project began in December~2025, when Gemini~3 Pro suggested the idea of using a coset-alignment lemma. We later found that a relevant version of the coset-alignment lemma appears in~\cite[Lemma 2.3]{Simpson1986ExactCoverings}. From January~2026 onward, we used ChatGPT to simplify and extend the arguments. ChatGPT supplied the current formulation of the cyclic dilation principle in Lemma~\ref{lemma: cyclic dilation}, which lies at the heart of our approach. At every stage, the authors independently checked, corrected, and developed the AI-generated material. The resulting arguments were then used as inputs for further prompts, creating an iterative process. The authors take full responsibility for all mathematical content in this preprint.

\section{Expressions for the Liminf and Limsup}\label{sec: Expressions for the Liminf and Limsup}
To assist with the proofs of the theorems in the introduction, we will prove convenient expressions for the $\liminf$ and the $\limsup$ of sequences $\{s_j(n)/n\}_n$ and $\{t_j(n)/n\}_n$ for every fixed $j\geq 1$. To write the expressions succinctly, we introduce some notation.

Let $j\geq 1$ be fixed. For $t\geq 2$, define the sets
\begin{align*}
    \mathcal{D}_P(j,t):=\{(p;q_2,\ldots,q_j):p\text{ prime},\ p<q_2<q_3<\cdots<q_j\leq t\}\quad\text{and}
\end{align*}
\begin{align*}
    \mathcal{D}_E(j,t):=\{(p;q_2,\ldots,q_j):p\text{ prime},\ p\leq q_2\leq q_3\leq\cdots\leq q_j\leq t\}.
\end{align*}
We note that when $j = 1$, the above sets collapse to

\begin{align*}
    \mathcal{D}_{P}(1, t)=\mathcal{D}_{E}(1,t) = \{p:p\text{ prime},\ p\leq t\}.
\end{align*}
Let $q_2,q_3,\ldots,q_j\geq 1$ be natural numbers. Set
\[
Q_k:=\prod_{m=k}^j q_m,\qquad Q_{j+1}:=1.
\]
Define, for $\varepsilon\in \{-1, 1\}$,
\[
A(q_2,\ldots,q_j; \varepsilon)
:=-\sum_{i=2}^{j+1}(-\varepsilon)^{j-i}Q_i.
\]
Notice that we have the recurrence
\[A(q_2, \dots, q_k;\varepsilon) = \varepsilon (1-q_kA(q_2,\dots, q_{k-1}; \varepsilon)). \]
For $d=(p;q_2,\ldots,q_j)$ put
\begin{equation} \label{eq:def-A}
m(d):=pQ_2=pq_2q_3\cdots q_j,
\quad
A_P(d):=A(q_2,\ldots,q_j; 1), \quad\text{and}\quad A_E(d):=A(q_2,\ldots,q_j; -1).    
\end{equation}

We follow the convention

\begin{equation} \label{eq:convention-A}
    m(\varnothing) = 1, \quad A_P(\varnothing) = 1 \quad\text{and}\quad A_E(\varnothing) = -1.    
\end{equation}

Let $\sigma\in\{P, E\}$. For a nonempty set $B\subseteq\mathcal D_\sigma(j,t)$, let
$ 1_{\CRT}(B;\sigma)$ be $1$ if
$$\gcd(m(d),m(e))
\mid A_\sigma(d)-A_\sigma(e)
\quad\text{for every }d,e\in B$$
and $0$ otherwise.

\begin{proposition}\label{propn: liminf expression pierce}
    For every fixed $j\geq 1$, we have
    \begin{align*}
        \liminf_{n\to\infty}\frac{s_j(n)}{n} = \sum_{t = j+1}^{\infty}\frac{\alpha_P(j, t)}{(t+1)(t+2)},
    \end{align*}
    where
    \begin{align*}
        \alpha_P(j, t):= \sum_{\varnothing\neq B\subseteq\mathcal{D}_P(j,t)}(-1)^{|B|+1}\frac{1}{\lcm_{d\in B}m(d)}.
    \end{align*}
    Additionally, the $\liminf$ is attained along the subsequence $(m!)_{m\geq 1}$.
\end{proposition}

\begin{proposition}\label{propn: limsup expression pierce}
    For every fixed $j\geq 1$, we have
    \begin{align*}
        \limsup_{n\to\infty}\frac{s_j(n)}{n} = \sum_{t = j+1}^{\infty}\frac{\beta_P(j, t)}{(t+1)(t+2)},
    \end{align*}
    where
    \begin{align*}
        \beta_P(j, t):= \sum_{\varnothing\neq B\subseteq\mathcal{D}_P(j,t)}(-1)^{|B|+1}\frac{1_{\CRT}(B; P)}{\lcm_{d\in B}m(d)}.
    \end{align*}
Additionally, the $\limsup$ is attained along the subsequence 
    $(m!+1)_{m\geq 1}$.
\end{proposition}

We have the following analogous results for the Engel-type iterated remainder sets $T_j(n)$ and their cardinalities $t_j(n)$.
\begin{proposition}\label{propn: liminf expression engel}
    For every fixed $j\geq 1$, we have
    \begin{align*}
        \liminf_{n\to\infty}\frac{t_j(n)}{n} = \sum_{t = 2}^{\infty}\frac{\alpha_E(j, t)}{t(t-1)},
    \end{align*}
    where
    \begin{align*}
        \alpha_E(j, t):= \sum_{\varnothing\neq B\subseteq\mathcal{D}_E(j,t)}(-1)^{|B|+1}\frac{1}{\lcm_{d\in B}m(d)}.
    \end{align*}
    Additionally, the $\liminf$ is attained along the sequence $(m!)_{m\geq 1}$.
\end{proposition}

\begin{proposition}\label{propn: limsup expression engel}
    For every fixed $j\geq 1$, we have
    \begin{align*}
        \limsup_{n\to\infty}\frac{t_j(n)}{n} = \sum_{t = 2}^{\infty}\frac{\beta_E(j, t)}{t(t-1)},
    \end{align*}
    where
    \begin{align*}
        \beta_E(j, t):= \sum_{\varnothing\neq B\subseteq\mathcal{D}_E(j,t)}(-1)^{|B|+1}\frac{1_{\CRT}(B; E)}{\lcm_{d\in B}m(d)}.
    \end{align*}
Additionally, the $\limsup$ is attained along the subsequence 
    $(m!+1)_{m\geq 1}$.
\end{proposition}
\section{Congruence Characterization of the Remainder Sets}\label{sec: Congruence Characterization of the Remainder Sets}
In this section, we will establish elementary congruence characterizations for the elements lying in certain segments of the iterated remainder sets. We begin by introducing some notation.

For \(t\geq 0\), put
\[
I_t(n):=\left(\frac{n}{t+2},\frac{n}{t+1}\right]\cap \mathbb Z\quad\text{and}\quad I^{\circ}_t(n):= \left(\frac{n}{t+2},\frac{n}{t+1}\right)\cap \mathbb Z.
\]
Additionally, for $t\geq 2$, define
\begin{align*}
    J_t(n):=\left(\frac{n}{t},\frac{n}{t-1}\right]\cap \mathbb Z\quad\text{and}\quad J^{\circ}_t(n):= \left(\frac{n}{t},\frac{n}{t-1}\right)\cap \mathbb Z.
\end{align*}

We have the following elementary characterization of $S_1(n)$.
\begin{lemma}\label{lemma: S1 characterization}
    For $0\leq u<n$, one has $u\in S_1(n)$ if and only if there is a prime $p$ such that
    \begin{align*}
        p\mid n-u\quad\text{and}\quad p(u+1)\leq n-u.
    \end{align*}
\end{lemma}
\begin{proof}
Indeed, if \(u\equiv n\bmod k\) with \(1\leq k\leq n/2\), then \(n-u=ck\) for some
integer \(c\geq 2\) and \(u<k\). Any prime divisor \(p\) of \(c\) satisfies the two displayed
conditions. Conversely, if these conditions hold, then \(k=(n-u)/p\) satisfies
\(u<k\leq n/2\), and hence \(u\equiv n\bmod k\).
\end{proof}
We provide an analogous characterization of $T_1(n)$ next.
\begin{lemma}\label{lemma: T1 characterization}
    For $0\leq u<n$, one has $u\in T_1(n)$ if and only if there is a prime $p$ such that
    \begin{align*}
        p\mid n+u\quad\text{and}\quad p(u+1)\leq n+u.
    \end{align*}
\end{lemma}
\begin{proof}
    Suppose $u\equiv (-n)\bmod{k}$ for some $1\leq k< n$, then $n = qk-u$
    for an integer $q\geq 2$ and $u<k$, and hence $n+u = qk$. Choose any prime $p$ dividing $q$. Since $u<k$, we have
    \begin{align*}
        p(u+1)\leq q(u+1)\leq qk = n+u.
    \end{align*}
    Conversely, suppose that there is a prime $p$ such that $p\mid n+u$ and $p(u+1)\leq n+u$. Then we can write $n+u = pk$ for some integer $k := (n+u)/p$. Then $k\geq u+1>u$. Also, $       k\leq (n+u)/2<n.
$    Thus $k\in T_0(n)$, and $n = pk-u,$
    which implies that $u\equiv (-n)\bmod{k}$, and hence $u\in T_1(n)$.

\end{proof}
Before stating the generalized congruence criteria, we prove an auxiliary lemma.

\comO{In the lemma below, can we use $N$ instead of $A$ maybe, if it dosen't hold a particular significance from before.}
\comS{I think here A makes sense as it is basically the same function}

\begin{lemma}\label{lem:backward-congruence-propagation}
Let \(j\geq 1\), let \(\varepsilon\in\{1,-1\}\), and let
\(q_1,q_2,\ldots,q_j\) be positive integers. Define
\[
A_{\varepsilon,1}:=\varepsilon,
\qquad
A_{\varepsilon,k}
:=
\varepsilon\bigl(1-q_kA_{\varepsilon,k-1}\bigr),
\qquad 2\leq k\leq j.
\]
Suppose that
\[
r_j\equiv A_{\varepsilon,j}n
\pmod{q_1q_2\cdots q_j}.
\]
For \(2\leq k\leq j\), define recursively
\[
r_{k-1}:=\frac{n-\varepsilon r_k}{q_k}.
\]
Then \(r_1,\ldots,r_{j-1}\) are integers and
\[
r_k\equiv A_{\varepsilon,k}n
\pmod{q_1q_2\cdots q_k},
\qquad 1\leq k\leq j.
\]
\end{lemma}

\begin{proof}
We argue by descending induction on \(k\). The asserted congruence for
\(k=j\) is the hypothesis. Suppose that \(2\leq k\leq j\) and that
\[
r_k\equiv A_{\varepsilon,k}n
\pmod{q_1q_2\cdots q_k}.
\]
Since
\[
A_{\varepsilon,k}\equiv\varepsilon\pmod{q_k},
\]
we have
\[
n-\varepsilon r_k\equiv0\pmod{q_k},
\]
so \(r_{k-1}\) is an integer. Moreover,
\[
1-\varepsilon A_{\varepsilon,k}
=
q_kA_{\varepsilon,k-1},
\]
and therefore
\[
n-\varepsilon r_k
\equiv
q_kA_{\varepsilon,k-1}n
\pmod{q_1q_2\cdots q_k}.
\]
Dividing by \(q_k\), we obtain
\[
r_{k-1}\equiv A_{\varepsilon,k-1}n
\pmod{q_1q_2\cdots q_{k-1}}.
\]
\end{proof}



\begin{lemma}\label{lem:congruence-criterion pierce}
Fix \(j\geq 1\) and \(t\geq j+1\). If $n\geq t(t+2)$, then for all integers
\(r\in I^{\circ}_t(n)\), we have $r\in S_j(n)$ if and only if there is a tuple \(d=(p;q_2,\ldots,q_j)\in \mathcal{D}_P(j,t)\) such that
\[
r\equiv A_P(d)n \pmod {m(d)}.
\]
Furthermore, for $r\in I_t(n)$, the above equivalence can fail only at the single endpoint $r = n/(t+1)$.

\end{lemma}
\begin{proof}
When \(j=1\), we have
\[
\mathcal D_P(1,t)=\{p\leq t:p\ \text{prime}\},
\qquad A_P(p)=1,\qquad m(p)=p.
\]
Thus the required congruence is \(r\equiv n\pmod p\). By
Lemma~\ref{lemma: S1 characterization}, membership \(r\in S_1(n)\)
implies \((p+1)r<n<(t+2)r\), and hence \(p\leq t\). Conversely, if
\(p\leq t\) and \(r\equiv n\pmod p\), then, for
\(r\in I_t^\circ(n)\), the positive integer
\[
n-(p+1)r
\]
is divisible by \(p\), and hence is at least \(p\). Therefore
\(p(r+1)\leq n-r\), so \(r\in S_1(n)\). The same argument works on
\(I_t(n)\), except possibly when \((t+1)r=n\). Hence the result holds
for \(j=1\).

Assume henceforth that \(j\geq2\).
Suppose first that \(r\in S_j(n)\). By the definition of \(S_j(n)\), there are integers $q_2,\ldots,q_j$ and integers \(r_1,\ldots,r_{j-1}\) with \(r_j=r\) and $r_1\in S_1(n)$ such that
\[
n=q_k r_{k-1}+r_k,
\qquad
0\leq r_k<r_{k-1},
\qquad
2\leq k\leq j,
\]
where \(r_j=r\). By the characterization of \(S_1(n)\) in Lemma~\ref{lemma: S1 characterization}, there is a prime
\(p\) such that
\begin{align*}
    p\mid n-r_1,
\qquad
p(r_1+1)\leq n-r_1.
\end{align*}

Repeated substitution in
\[
n=q_k r_{k-1}+r_k,
\qquad
2\leq k\leq j,
\]
gives
\[
r_1
=
\frac{
n\sum_{i=2}^{j}(-1)^iQ_{i+1}
+
(-1)^{j-1}r
}{Q_2}.
\]
Since \(p\mid n-r_1\), this identity gives
\[
r\equiv A_P(d)n\pmod{pQ_2},
\]
where \(d=(p;q_2,\ldots,q_j)\). This is precisely
\[
r\equiv A_P(d)n\pmod{m(d)}.
\]



It remains to determine the possible values of
\(p,q_2,\ldots,q_j\). For \(2\leq k<j\), the relations
\[
n=q_kr_{k-1}+r_k>(q_k+1)r_k,
\qquad
q_{k+1}=\left\lfloor\frac{n}{r_k}\right\rfloor,
\]
show that \(q_{k+1}>q_k\). Moreover, since $r\in I_t^{\circ}(n)$,
\[
(q_j+1)r<q_jr_{j-1}+r = n<(t+2)r,
\]
so \(q_j\leq t\). Finally,
\[
n-r_1=(q_2-1)r_1+r_2.
\]
If \(p\geq q_2\), then
\[
p(r_1+1)\geq q_2(r_1+1)>n-r_1,
\]
contrary to Lemma~\ref{lemma: S1 characterization}. Hence
\[
p<q_2<\cdots<q_j\leq t,
\]
and therefore \(d\in\mathcal D_P(j,t)\).

Conversely, suppose that there is a tuple
\[
d=(p;q_2,\ldots,q_j)\in\mathcal D_P(j,t)
\]
such that
\[
r\equiv A_P(d)n\pmod{m(d)}.
\]
Define \(q_1:=p\) and \(r_j:=r\). Recursively set
\[
r_{k-1}:=\frac{n-r_k}{q_k},
\qquad
1\leq k\leq j.
\]


It follows from Lemma~\ref{lem:backward-congruence-propagation}, applied with
\(\varepsilon=1\), that all of \( r_1,\ldots,r_{j-1}\) are integers and
$r_1\equiv n\pmod{q_1}.$
Hence $r_0$
is also an integer.


We next verify the inequalities between successive remainders. Since
\(r\in I_t^\circ(n)\) and \(q_j\leq t\), we have
\[
n>(t+1)r\geq(q_j+1)r.
\]
Therefore
\[
r_{j-1}-r_j
=
\frac{n-(q_j+1)r_j}{q_j}
>0.
\]
Thus \(r_{j-1}>r_j>0\).

Now suppose that \(3\leq k\leq j\) and
\[
r_{k-1}>r_k>0.
\]
Since \(q_{k-1}<q_k\), we have
\[
n=q_kr_{k-1}+r_k
>
(q_{k-1}+1)r_{k-1}.
\]
Consequently,
\[
r_{k-2}-r_{k-1}
=
\frac{n-(q_{k-1}+1)r_{k-1}}{q_{k-1}}
>0.
\]
Iterating backwards gives
\[
0<r_j<r_{j-1}<\cdots<r_1.
\]

It remains to verify the inequality required for \(r_1\in S_1(n)\). Since
\[
n-r_1=(q_2-1)r_1+r_2
\]
and \(p<q_2\),
\[
\begin{aligned}
n-r_1-p(r_1+1)
&=(q_2-1-p)r_1+r_2-p\\
&\geq r_2-p.
\end{aligned}
\]
Now \(n\geq t(t+2)\) and \(n<(t+2)r\), so
\[
r>\frac{n}{t+2}\geq t.
\]
Since \(r\) is an integer,
\[
r\geq t+1.
\]
Also \(r_2\geq r\), while \(p<q_2\leq t\), and hence
\[
r_2-p\geq r-p>0.
\]
Therefore
\[
p(r_1+1)\leq n-r_1.
\]
The characterization of \(S_1(n)\) in Lemma~\ref{lemma: S1 characterization} now gives \(r_1\in S_1(n)\), and the
recursive relations imply
\[
r=r_j\in S_j(n).
\]

The proof remains valid for \(r\in I_t(n)\) unless
\((t+1)r=n\). At this endpoint, failure is possible only when
\(q_j=t\), in which case
\[
r_{j-1}=\frac{n-r}{t}=r.
\]

\end{proof}

We state a similar lemma for the Engel-type iterated remainder sets \(T_j(n)\).
\begin{lemma}\label{lem:congruence-criterion engel}
Fix \(j\geq 1\) and \(t\geq 2\). Let $n\geq 1$. For all integers
\(r\in J^{\circ}_t(n)\), we have $r\in T_j(n)$ if and only if there is a tuple \(d=(p;q_2,\ldots,q_j)\in \mathcal{D}_E(j,t)\) such that
\[
r\equiv A_E(d)n \pmod {m(d)}.
\]
Furthermore, for $r\in J_t(n)$, the above equivalence can fail only at the single endpoint $r = n/(t-1)$.
\end{lemma}
\begin{proof}
The proof is similar to that of Lemma~\ref{lem:congruence-criterion pierce}. We reproduce the argument for completeness.

When \(j=1\), we have
\[
\mathcal D_E(1,t)=\{p\leq t:p\ \text{prime}\},
\qquad A_E(p)=-1,\qquad m(p)=p.
\]
Thus the required congruence is \(r\equiv (-n)\pmod p\). By
Lemma~\ref{lemma: T1 characterization}, membership \(r\in T_1(n)\)
implies \((p-1)r<n<tr\), and hence \(p\leq t\). Conversely, if
\(p\leq t\) and \(r\equiv (-n)\pmod p\), then, for
\(r\in J_t^\circ(n)\), the positive integer
\[
n-(p-1)r
\]
is divisible by \(p\), and hence is at least \(p\). Therefore
\(p(r+1)\leq n+r\), so \(r\in T_1(n)\). The same argument works on
\(J_t(n)\), except possibly when \((t-1)r=n\). Hence the result holds
for \(j=1\).

Assume henceforth that \(j\geq2\).
Suppose first that \(r\in T_j(n)\). By the definition of \(T_j(n)\), there are integers $q_2,\ldots,q_j$ and integers \(r_1,\ldots,r_{j-1}\) with \(r_j=r\) and $r_1\in T_1(n)$ such that
\[
n=q_k r_{k-1}-r_k,
\qquad
0\leq r_k<r_{k-1},
\qquad
2\leq k\leq j,
\]
where \(r_j=r\). By the characterization of \(T_1(n)\) in Lemma~\ref{lemma: T1 characterization}, there is a prime
\(p\) such that
\begin{align*}
    p\mid n+r_1,
\qquad
p(r_1+1)\leq n+r_1.
\end{align*}

Repeated substitution in
\[
n=q_k r_{k-1}-r_k,
\qquad
2\leq k\leq j,
\]
gives
\[
r_1
=
\frac{
n\sum_{i=3}^{j+1}Q_i+r
}{Q_2}.
\]
Consequently,
\[
Q_2(n+r_1)
=
n\sum_{i=2}^{j+1}Q_i+r.
\]
Since \(p\mid n+r_1\), this identity gives
\[
r\equiv A_E(d)n\pmod{pQ_2},
\]
where \(d=(p;q_2,\ldots,q_j)\) and
\[
A_E(d)=-\sum_{i=2}^{j+1}Q_i.
\]
This is precisely
\[
r\equiv A_E(d)n\pmod{m(d)}.
\]



It remains to determine the possible values of
\(p,q_2,\ldots,q_j\). Since
\[
q_k=\left\lceil\frac{n}{r_{k-1}}\right\rceil
\quad\text{and}\quad
r_k<r_{k-1},
\]
we have
\[
q_2\leq q_3\leq\cdots\leq q_j.
\]
Furthermore, as $r\in J_t^{\circ}(n)$,
\[
(q_j-1)r<q_jr_{j-1}-r=n<tr,
\]
so \(q_j\leq t\). Finally,
\[
n+r_1=(q_2+1)r_1-r_2.
\]
If \(p\geq q_2+1\), then
\[
p(r_1+1)
\geq
(q_2+1)(r_1+1)
>
n+r_1,
\]
contrary to Lemma~\ref{lemma: T1 characterization}. Thus
\[
p\leq q_2\leq\cdots\leq q_j\leq t,
\]
and hence
\[
d=(p;q_2,\ldots,q_j)\in\mathcal D_E(j,t).
\]

Conversely, suppose that there is a tuple
\[
d=(p;q_2,\ldots,q_j)\in\mathcal D_E(j,t)
\]
such that
\[
r\equiv A_E(d)n\pmod{m(d)}.
\]
Define
\[
q_1:=p,
\qquad
r_j:=r,
\]
and recursively set
\[
r_{k-1}:=\frac{n+r_k}{q_k},
\qquad
1\leq k\leq j.
\]

It follows from Lemma~\ref{lem:backward-congruence-propagation}, applied with
\(\varepsilon=-1\), that all of \( r_1,\ldots,r_{j-1}\) are integers and
$r_1\equiv -n\pmod{q_1}.$
Hence $r_0$
is also an integer.



We next verify the inequalities between successive remainders. Since
\(r\in J_t^\circ(n)\) and \(q_j\leq t\), we have
\[
n>(t-1)r\geq(q_j-1)r.
\]
Therefore
\[
r_{j-1}-r_j
=
\frac{n-(q_j-1)r_j}{q_j}
>0.
\]
Thus \(r_{j-1}>r_j>0\).

Now suppose that \(2\leq k\leq j\) and
\[
r_{k-1}>r_k>0.
\]
Since \(q_{k-1}\leq q_k\), we have
\[
n=q_kr_{k-1}-r_k
>
(q_k-1)r_{k-1}
\geq
(q_{k-1}-1)r_{k-1}.
\]
Consequently,
\[
r_{k-2}-r_{k-1}
=
\frac{n-(q_{k-1}-1)r_{k-1}}{q_{k-1}}
>0.
\]
For \(k=2\), this argument uses
\[
q_1=p\leq q_2.
\]
Iterating backwards gives
\[
0<r_j<r_{j-1}<\cdots<r_1<r_0.
\]

Furthermore,
\[
n=pr_0-r_1.
\]
Since \(p\geq2\) and \(r_1<r_0\),
\[
n-r_0
=
(p-1)r_0-r_1
\geq
r_0-r_1
>0.
\]
Thus
\[
1\leq r_0<n,
\]
so \(r_0\in T_0(n)\). The recursive relations
\[
n=q_kr_{k-1}-r_k,
\qquad
0\leq r_k<r_{k-1},
\qquad
1\leq k\leq j,
\]
show that
\[
r_k=(-n)\bmod r_{k-1}
\]
for every \(1\leq k\leq j\). Therefore
\[
r=r_j\in T_j(n).
\]

The proof remains valid for \(r\in J_t(n)\) unless
\((t-1)r=n\). At this endpoint, failure is possible only when
\(q_j=t\), in which case
\[
r_{j-1}=\frac{n+r}{t}=r.
\]

\end{proof}

Let $\sigma\in\{E, P\}$. For fixed \(j,t\), define
\[
M_{\sigma}(j,t):=\operatorname{lcm}\{m(d):d\in \mathcal{D}_\sigma(j,t)\}.
\]
If \(\mathcal{D}_\sigma(j,t)=\varnothing\), set \(M_\sigma(j,t)=1\). For \(a\in \mathbb Z/M_\sigma(j,t)\mathbb Z\), define
\[
U_{j,t}^{(\sigma)}(a):=
\bigcup_{d\in \mathcal{D}_\sigma(j,t)}
\left\{
x\in \mathbb Z/M_\sigma(j,t)\mathbb Z:
x\equiv A_\sigma(d)a \pmod {m(d)}
\right\},
\]
and put
\[
D_{j,t}^{(\sigma)}(a):=\frac{|U_{j,t}^{(\sigma)}(a)|}{M_\sigma(j,t)}.
\]

\begin{lemma}\label{lem:interval-count pierce}
For fixed \(j\geq 1\) and \(t\geq j+1\), uniformly as \(n\to\infty\),
\[
\frac{|S_j(n)\cap I_t(n)|}{n}
=
\frac{D_{j,t}^{(P)}(n\bmod M_P(j,t))}{(t+1)(t+2)}
+O_{j,t}\left(\frac1n\right).
\]
\end{lemma}

\begin{proof}
By Lemma~\ref{lem:congruence-criterion pierce}, up to \(O_{j,t}(1)\) endpoint exceptions, the
integers \(r\in S_j(n)\cap I_t(n)\) are exactly those lying in the union of residue classes
\[
r\equiv A_P(d)n \pmod {m(d)},\qquad d\in \mathcal{D}_P(j,t).
\]
This union is periodic modulo \(M_P(j,t)\). Hence
\[
\left|S_j(n)\cap I_t(n)\right|
=
D_{j,t}^{(P)}(n\bmod M_P(j,t))\,\frac{n}{(t+1)(t+2)}
+O_{j,t}(1).
\]
Dividing by \(n\) gives the claimed estimate.
\end{proof}

\begin{lemma}\label{lem:interval-count engel}
For fixed \(j\geq 1\) and \(t\geq 2\), uniformly as \(n\to\infty\),
\[
\frac{|T_j(n)\cap J_t(n)|}{n}
=
\frac{D_{j,t}^{(E)}(n\bmod M_E(j,t))}{t(t-1)}
+O_{j,t}\left(\frac1n\right).
\]
\end{lemma}

\begin{proof}
The proof is similar to that of Lemma~\ref{lem:interval-count pierce}, relying on Lemma~\ref{lem:congruence-criterion engel} instead.
\end{proof}

\section{The Cyclic Dilation Lemma}\label{sec: The Cyclic Dilation Lemma}
To make further progress, we need the following elementary lemma.
\begin{lemma}\label{lemma: cyclic dilation}
Let $M\geq 1$ and $\mathcal{I}$ be a finite index set. Suppose that $m_i\mid M$ are positive divisors and $b_i\in\mathbb Z$ for $i\in \mathcal{I}$. For
\[
U(a):=\bigcup_{i\in\mathcal{I}}\{x\in\mathbb Z/M\mathbb Z:x\equiv ab_i\pmod {m_i}\},
\]
we have
\[
|U(0)|\leq |U(a)|\leq |U(u)|
\]
for every $a\in\mathbb Z/M\mathbb Z$ and every unit
$u\in(\mathbb Z/M\mathbb Z)^\times$. Moreover, all units give the same
cardinality.
\end{lemma}

\begin{proof}
First we prove the following claim. If $\ell\mid M$ is prime, then
\[
|U(\ell c)|\leq |U(c)|
\]
for every $c\in\mathbb Z/M\mathbb Z$.

Write
\[
M=\ell^K N,
\qquad (\ell,N)=1,
\]
and
\[
m_i=\ell^{e_i}n_i,
\qquad n_i\mid N.
\]
By the Chinese remainder theorem, a residue modulo $M$ may be written as a
pair of residues modulo $\ell^K$ and modulo $N$.

Fix a residue $r$ modulo $N$. For $U(c)$, the congruence modulo $N$ selects
exactly those indices $i$ for which
\[
r\equiv cb_i\pmod {n_i}.
\]
Denote the set of such indices by $I$. For $U(\ell c)$, the corresponding residue modulo $N$ is $\ell r$, and the
condition becomes
\[
\ell r\equiv \ell cb_i\pmod {n_i}.
\]
Since multiplication by $\ell$ is invertible modulo $N$, these are the same
conditions. Thus it remains only to compare, modulo $\ell^K$, the unions of
the residue classes

\comO{Can we use $R_i$ for these residue classes, I think they are not related to the function A, right?}
\comS{Yes, I changed it}

\[
R_i(c)=\{y\in\mathbb Z/\ell^K\mathbb Z:y\equiv cb_i\pmod{\ell^{e_i}}\}
\]
and
\[
R_i(\ell c)=\{y\in\mathbb Z/\ell^K\mathbb Z:y\equiv \ell cb_i
\pmod{\ell^{e_i}}\}.
\]

For residue classes modulo powers of $\ell$, any two classes are either
disjoint, or one is contained in the other. Starting from the family \(\{R_i(c)\}_{i\in I}\), discard every class that is properly contained in another one, and retain only one representative whenever several of the remaining classes coincide. Let $J$ be the
set of remaining indices. Then the classes $R_j(c)$ with $j\in J$ are
pairwise disjoint, and their union is the same as the union of all the classes
$\{R_i(c)\}_{i\in I}$.

Now suppose that
\[
R_i(c)\subseteq R_j(c).
\]
Then $e_i\geq e_j$ and
\[
cb_i\equiv cb_j\pmod{\ell^{e_j}}.
\]
Multiplying by $\ell$ gives
\[
\ell cb_i\equiv \ell cb_j\pmod{\ell^{e_j}},
\]
and hence
\[
R_i(\ell c)\subseteq R_j(\ell c).
\]
Therefore, after replacing $c$ by $\ell c$, every deleted class is still
contained in one of the classes that remained.

Thus
\[
\bigcup_{i\in I} R_i(\ell c)\subseteq \bigcup_{j\in J} R_j(\ell c).
\]
The classes $R_j(\ell c)$ have the same sizes as the corresponding
$R_j(c)$. Hence
\[
\left|\bigcup_{i\in I} R_i(\ell c)\right|
\leq
\sum_{j\in J}|R_j(\ell c)|
=
\sum_{j\in J}|R_j(c)|
=
\left|\bigcup_{i\in I} R_i(c)\right|.
\]
Summing the above over the residues modulo $N$ gives
\[
|U(\ell c)|\leq |U(c)|.
\]

Now let $a\in\mathbb Z/M\mathbb Z$. Choose a divisor $q\mid M$ and a unit
$v\in(\mathbb Z/M\mathbb Z)^\times$ such that
\[
a\equiv qv\pmod M.
\]
This is obtained by putting into $q$ exactly the prime-power factors of $M$
that also divide $a$, up to the powers occurring in $M$.

Applying the claim one prime factor at a time gives
\[
|U(a)|=|U(qv)|\leq |U(v)|.
\]
Applying the same claim further, multiplying by the remaining prime factors
needed to get from $q$ to $M$, gives
\[
|U(0)|=|U(Mv)|\leq |U(qv)|=|U(a)|.
\]
Thus
\[
|U(0)|\leq |U(a)|\leq |U(v)|.
\]

Finally, if $u$ is a unit modulo $M$, multiplication by $u^{-1}$ is a
bijection from $U(u)$ to $U(1)$. Hence all units give the same cardinality.
Therefore $|U(v)|=|U(u)|$ for every unit $u$, and the lemma follows.
\end{proof}

For $T\geq j+1$, put
\[
L_P(j,T):=\lcm_{j+1\leq t\leq T}M_P(j,t).
\]
Additionally, for $T\geq 2$, put
\[
L_E(j, T):=\lcm_{2\leq t\leq T}M_E(j, t).
\]

For $a\in\Z/L_P(j,T)\Z$, define
\[
F_{j,T}^{(P)}(a):=
\sum_{t=j+1}^{T}
\frac{D_{j,t}^{(P)}(a\bmod M_P(j,t))}{(t+1)(t+2)}.
\]
Analogously, for $a\in \Z/L_E(j,T)\Z$, define
\[
F_{j,T}^{(E)}(a):=
\sum_{t=2}^{T}
\frac{D_{j,t}^{(E)}(a\bmod M_E(j,t))}{t(t-1)}.
\]
The tails satisfy the uniform bound
\[
\sum_{t>T}|S_j(n)\cap I_t(n)|\leq \frac{n}{T+2}+1, \quad \text{and}
\]
\[
\sum_{t>T}|T_j(n)\cap J_t(n)|\leq \frac{n}{T}+1
\]

We will now show that $S_j(n)\cap I_t(n)$ is empty for $t<j+1$. Indeed, we claim by induction on \(j\geq 1\) that every nonzero
\(r\in S_j(n)\) satisfies
\[
r<\frac{n}{j+2}.
\]
For \(j=1\), write \(r\equiv n\bmod k\) with \(k\in S_0(n)\), so
\(k\leq n/2\). If \(q=\lfloor n/k\rfloor\), then \(q\geq 2\) and
\(k>n/(q+1)\); hence
\[
r=n-qk<\frac{n}{q+1}\leq \frac{n}{3}.
\]
Now suppose the claim holds for \(j-1\). If \(r=n\bmod k\) with
\(k\in S_{j-1}(n)\setminus\{0\}\), then \(k<n/(j+1)\), so
\(q:=\lfloor n/k\rfloor\geq j+1\). Since \(k>n/(q+1)\),
\[
r=n-qk<\frac{n}{q+1}\leq \frac{n}{j+2}.
\]
Thus, if \(t<j+1\), then \(I_t(n)\subseteq [n/(j+1),n]\), and therefore
\(S_j(n)\cap I_t(n)=\varnothing\).

Hence, using the fact that, for a fixed $T$, there are infinitely many $n$ congruent to a given residue class modulo $L_P(j, T)$, we have
\[
\liminf_{n\to\infty}\frac{s_j(n)}{n}=
\lim_{T\to\infty}
\min_{a\in\Z/L_P(j,T)\Z}F_{j,T}^{(P)}(a),\quad\limsup_{n\to\infty}\frac{s_j(n)}{n}=
\lim_{T\to\infty}
\max_{a\in\Z/L_P(j,T)\Z}F_{j,T}^{(P)}(a),\quad\text{and} 
\]
\[
\liminf_{n\to\infty}\frac{t_j(n)}{n}=
\lim_{T\to\infty}
\min_{a\in\Z/L_E(j,T)\Z}F_{j,T}^{(E)}(a),\quad\limsup_{n\to\infty}\frac{t_j(n)}{n}=
\lim_{T\to\infty}
\max_{a\in\Z/L_E(j,T)\Z}F_{j,T}^{(E)}(a).
\]
\subsection{Proof of Propositions~\ref{propn: liminf expression pierce} and~\ref{propn: liminf expression engel}}
By Lemma~\ref{lemma: cyclic dilation}, each $D_{j,t}^{(\sigma)}(a)$ is minimized at $a=0$. Therefore
\[
\liminf_{n\to\infty}s_j(n)/n=
\sum_{t=j+1}^{\infty}\frac{\alpha_P(j,t)}{(t+1)(t+2)}\quad\text{and}
\]
\[
\liminf_{n\to\infty}t_j(n)/n=
\sum_{t=2}^{\infty}\frac{\alpha_E(j,t)}{t(t-1)},
\]
where for $\sigma\in \{P, E\}$,
\[
\alpha_\sigma(j,t):=D_{j,t}^{(\sigma)}(0)
\]
or equivalently,
\[
\alpha_\sigma(j,t)
=
\sum_{\varnothing\neq B\subseteq\mathcal D_\sigma(j,t)}
(-1)^{|B|+1}
\frac1{\lcm_{d\in B}m(d)}.
\]
Let $\sigma\in \{P, E\}$. For each fixed \(T\), the integer \(L_{\sigma}(j,T)\) is fixed. Hence, for every
\(m\geq L_{\sigma}(j,T)\), we have
\[
L_{\sigma}(j,T)\mid m!,
\]
and therefore
\[
m!\equiv 0 \pmod{L_{\sigma}(j,T)}.
\]
Consequently,
\[
\liminf_{n\to \infty}\frac{s_j(n)}{n} = \lim_{m\to\infty}\frac{s_j(m!)}{m!}\quad\text{and}\quad \liminf_{n\to \infty}\frac{t_j(n)}{n} = \lim_{m\to\infty}\frac{t_j(m!)}{m!}.
\]

\subsection{Proof of Propositions~\ref{propn: limsup expression pierce} and~\ref{propn: limsup expression engel}}
Again, by Lemma~\ref{lemma: cyclic dilation}, each $D_{j,t}^{(\sigma)}(a)$ is maximized when $a$ is a unit. 
Therefore
\[
\limsup_{n\to\infty}s_j(n)/n=
\sum_{t=j+1}^{\infty}\frac{\beta_P(j,t)}{(t+1)(t+2)}\quad\text{and}
\]
\[
\limsup_{n\to\infty}t_j(n)/n=
\sum_{t=2}^{\infty}\frac{\beta_E(j,t)}{t(t-1)},
\]
where $\sigma\in \{P, E\}$.
\begin{align}\label{eqn: beta alternate defn}
    \beta_\sigma(j,t):=D_{j,t}^{(\sigma)}(1)
\end{align}
or equivalently,
\[
\beta_\sigma(j,t)
=
\sum_{\varnothing\neq B\subseteq\mathcal D_\sigma(j,t)}
(-1)^{|B|+1}
\frac{1_{\CRT}(B;\sigma)}{\lcm_{d\in B}m(d)}.
\]
Let $\sigma\in \{P, E\}$. For each fixed \(T\), the integer \(L_{\sigma}(j,T)\) is fixed. Hence, for every
\(m\geq L_{\sigma}(j,T)\), we have
\[
L_{\sigma}(j,T)\mid m!,
\]
and therefore
\[
m!+1\equiv 1 \pmod{L_{\sigma}(j,T)}.
\]
Consequently,
\[
\limsup_{n\to \infty}\frac{s_j(n)}{n} = \lim_{m\to\infty}\frac{s_j(m!+1)}{m!+1}\quad\text{and}\quad \limsup_{n\to \infty}\frac{t_j(n)}{n} = \lim_{m\to\infty}\frac{t_j(m!+1)}{m!+1}.
\]
\section{Characterizing the Convergence}\label{sec: Characterizing the Convergence}
The following propositions prove the separation between $\liminf$ and $\limsup$ in the Pierce and Engel cases, respectively.
\begin{proposition}\label{propn: limit non-existence pierce}
    For every fixed $j\geq 2$,
    \begin{align*}
        \beta_P(j, j+2)-\alpha_P(j, j+2)\geq \frac{j-1}{(j+2)!}.
    \end{align*}
    Moreover, for every $t\geq j+1$, we have
    \begin{align*}
        \beta_P(j, t)\geq \alpha_P(j, t).
    \end{align*}
\end{proposition}
Combining the above proposition with Propositions \ref{propn: liminf expression pierce} and \ref{propn: limsup expression pierce} completes the proof of nonexistence of limit claim of Theorem \ref{thm:non-conv-s}.
\begin{proposition}\label{propn: limit non-existence engel}
    For every \(j\geq2\),
\[
\beta_E(j,3)-\alpha_E(j,3)\geq\frac1{4\cdot3^j}.
\]
Moreover, for every \(t\geq2\),
\[
\beta_E(j,t)\geq\alpha_E(j,t).
\]
\end{proposition}
Combining the above proposition with Propositions \ref{propn: liminf expression engel} and \ref{propn: limsup expression engel} completes the proof of the nonexistence of limit claim of Theorem \ref{thm:non-conv-t} in the case when $j\geq 2$.

When $j = 1$, we have the following result from which the existence of the limit $\lim_{n\to\infty} t_1(n)/n$ and its value follow immediately.
\begin{proposition}\label{propn: limit for engel j = 1}
    We have that for every $t\geq 2$,
    \begin{align*}
        \beta_{E}(1, t) = \alpha_{E}(1, t) = 1-\prod_{\substack{p\leq t\\ p \text{ prime}}}\left(1-\frac{1}{p}\right)
    \end{align*}
    Additionally, we have the identity
    \begin{align*}
        \sum_{t = 2}^{\infty}\frac{1}{t(t-1)}\left(1-\prod_{\substack{p\leq t\\ p \text{ prime}}}\left(1-\frac{1}{p}\right)\right) = \sum_{p \text{ prime}}\frac{1}{p(p-1)}\prod_{\substack{q<p\\ q \text{ prime}}}\left(1-\frac{1}{q}\right).
    \end{align*}

\end{proposition}
We shall also prove the following quantitative bounds, finishing the proofs of Theorem~\ref{thm:non-conv-s} and Theorem~\ref{thm:non-conv-t}.
\begin{proposition}\label{propn: bounds on limsup and liminf}
    For every $j\geq 2$, we have
    \[
        \frac{1}{(j+2)!}\leq \liminf_{n\to\infty}\frac{s_j(n)}{n}\leq \limsup_{n\to\infty}\frac{s_j(n)}{n}\leq \exp\left(-\frac{j}{e}+O(\log j)\right),
    \]
    where the implicit constant is absolute and independent of $j$.
\end{proposition}

\begin{proposition}\label{propn: bounds on engel limsup and liminf}
For every \(j\geq 2\), we have
\[
\frac{3}{2^{j+1}}-\frac{1}{2\cdot 3^j}
\leq
\liminf_{n\to\infty}\frac{t_j(n)}{n}
\leq
\limsup_{n\to\infty}\frac{t_j(n)}{n}
<
2^{(1-j)/2}.
\]
\end{proposition}

\subsection{Proof of Proposition~\ref{propn: limit for engel j = 1}}

Put
\[
P_t:=\prod_{\substack{p\leq t\\ p\ \mathrm{prime}}}p.
\]
Since
\[
\mathcal D_E(1,t)=\{p\leq t:p\ \mathrm{prime}\},
\qquad
m(p)=p,
\qquad
A_E(p)=-1,
\]
we have
\[
U_{1,t}^{(E)}(0)
=
\bigcup_{\substack{p\mid P_t\\ p\text{ prime}}}
\{x\in\mathbb Z/P_t\mathbb Z:x\equiv0\pmod p\}.
\]
Its complement consists of the units modulo \(P_t\), and therefore
\[
\alpha_E(1,t)
=
1-\frac{\varphi(P_t)}{P_t}
=
1-\prod_{\substack{p\leq t\\ p\ \mathrm{prime}}}
\left(1-\frac1p\right).
\]
Moreover, translation by \(1\) maps \(U_{1,t}^{(E)}(1)\) bijectively onto
\(U_{1,t}^{(E)}(0)\), so
\[
\beta_E(1,t)=\alpha_E(1,t).
\]

\comO{In the equation below can we use the $F_t$ rather than $F_t$ asssuming it dosen't hold any significance or relation with previous expressions.}
\comS{Yes, changed it}

Set
\[
F_t:=
1-\prod_{\substack{p\leq t\\ p\ \mathrm{prime}}}
\left(1-\frac1p\right).
\]
Summation by parts gives
\[
\sum_{t=2}^{\infty}\frac{F_t}{t(t-1)}
=
F_2+\sum_{t=2}^{\infty}\frac{F_{t+1}-F_t}{t}.
\]
Now \(F_{t+1}=F_t\) unless \(t+1=p\) is prime, in which case
\[
F_p-F_{p-1}
=
\frac1p
\prod_{\substack{q<p\\q\ \mathrm{prime}}}
\left(1-\frac1q\right).
\]
Since \(F_2=1/2\), the term corresponding to \(p=2\) may be included in
the resulting prime sum, giving
\[
\sum_{t=2}^{\infty}\frac{F_t}{t(t-1)}
=
\sum_{p\ \mathrm{prime}}
\frac{1}{p(p-1)}
\prod_{\substack{q<p\\q\ \mathrm{prime}}}
\left(1-\frac1q\right).
\]
\subsection{Proof of Proposition~\ref{propn: limit non-existence pierce}}
Recall, for \(j\ge 2\) and \(q_2,\ldots,q_j\in\Z\),
\[
A_P(q_2,\ldots,q_j)
:=1-q_j+q_{j-1}q_j-\cdots+(-1)^{j-1}q_2q_3\cdots q_j .
\]
Thus, if \(\mathbf q=(q_2,\ldots,q_j)\) and \(q\) is appended at the end, then
\begin{align}\label{eqn: recurrence for A}
A_P(\mathbf q,q)=1-qA_P(\mathbf q).
\end{align}

\begin{lemma}\label{lemma: tail stripping}
Let \(\mathbf u,\mathbf v\) be two finite integer tuples, and let
\(\mathbf t=(t_1,\ldots,t_r)\) be a common tail. Then
\[
A_P(\mathbf u,\mathbf t)-A_P(\mathbf v,\mathbf t)
=
(-1)^r\left(\prod_{\ell=1}^r t_\ell\right)
\bigl(A_P(\mathbf u)-A_P(\mathbf v)\bigr).
\]
\end{lemma}

\begin{proof}
The recurrence in~\eqref{eqn: recurrence for A} gives
\[
A_P(\mathbf u,q)-A_P(\mathbf v,q)
=
\bigl(1-qA_P(\mathbf u)\bigr)-\bigl(1-qA_P(\mathbf v)\bigr)
=
-q\bigl(A_P(\mathbf u)-A_P(\mathbf v)\bigr).
\]
Iterating this identity over \(t_1,\ldots,t_r\) gives the stated formula.
\end{proof}

\begin{lemma}\label{lemma: CRT incompatibility}
Let \(J\ge 4\). For \(2\le a\le J\), define the integer tuples
\[
\mathbf q_a^{(J)}
:=
\begin{cases}
(4,5,\ldots,J), & a=2,3,\\
(3,4,\ldots,\widehat a,\ldots,J), & 4\le a\le J.
\end{cases}
\]
Then, for \(2\le a<b\le J\),
\[
\gcd(J!/a,J!/b)\mid A_P(\mathbf q_a^{(J)})-A_P(\mathbf q_b^{(J)})
\]
if and only if \((a,b)=(2,3)\).
\end{lemma}

\begin{proof}
First, if \((a,b)=(2,3)\), then
\[
\mathbf q_2^{(J)}=\mathbf q_3^{(J)}=(4,5,\ldots,J),
\]
so \(A_P(\mathbf q_2^{(J)})=A_P(\mathbf q_3^{(J)})\), and the divisibility is immediate.

Now suppose \(b\ge 4\). Define the reduced shifts
\[
B_{a,b}
:=
\begin{cases}
A_P(4,5,\ldots,b), & a=2,3,\\
A_P(3,4,\ldots,\widehat a,\ldots,b), & 4\le a<b,
\end{cases}
\]
and
\[
B_{b,b}:=A_P(3,4,\ldots,b-1).
\]
The strings \(\mathbf q_a^{(J)}\) and \(\mathbf q_b^{(J)}\) have common tail
\[
(b+1,b+2,\ldots,J).
\]
By Lemma~\ref{lemma: tail stripping},
\[
A_P(\mathbf q_a^{(J)})-A_P(\mathbf q_b^{(J)})
=
\pm
\left(\prod_{r=b+1}^J r\right)
\bigl(B_{a,b}-B_{b,b}\bigr).
\]
Also,
\[
\gcd\left(\frac{J!}{a},\frac{J!}{b}\right)
=
\frac{J!}{\operatorname{lcm}(a,b)}
=
\left(\prod_{r=b+1}^J r\right)
\frac{b!}{\operatorname{lcm}(a,b)}.
\]
Hence it suffices to show that the divisibility condition 
\[
\frac{b!}{\operatorname{lcm}(a,b)}
\mid
B_{a,b}-B_{b,b}
\]
never happens.

First note that, for \(b\ge4\),
\[
\frac{b!}{\operatorname{lcm}(a,b)}
\]
is even. Thus it is enough, in most cases, to show that
\(B_{a,b}-B_{b,b}\) is odd.

First consider the case \(b=4\). Then, for \(a=2,3\),
\[B_{a, 4}-B_{4, 4} = A_P(4)-A_P(3) = -1.\]
Thus the required divisibility fails for \(b=4\).

Assume henceforth that \(b\ge5\). From the recurrence~\eqref{eqn: recurrence for A} it follows that 
if an integer tuple ends in an even number, then its \(A\)-value is odd, whereas,
if it ends in an odd number immediately preceded by an even number, then its \(A\)-value is even.
Therefore
\[
B_{b,b}=A_P(3,4,\ldots,b-1)
\]
is even when \(b\) is even and odd when \(b\) is odd. On the other hand,
unless \(a=b-1\) with \(b\) odd, the string defining \(B_{a,b}\) has the opposite parity.
Consequently
\[
B_{a,b}-B_{b,b}
\]
is odd, and so the above even divisor cannot divide it.

It remains to handle the exceptional case \(a=b-1\) with \(b\) odd. 

Observe that
\[
B_{b-1,b}=A_P(3,4,\ldots,b-2,b)
=1-bA_P(3,4,\ldots,b-2),
\]
while
\[
B_{b,b}=A_P(3,4,\ldots,b-2,b-1)
=1-(b-1)A_P(3,4,\ldots,b-2).
\]
Thus
\[
B_{b-1,b}-B_{b,b}
=
-A_P(3,4,\ldots,b-2).
\]
Since
\[
A_P(3,4,\ldots,b-2)=1-(b-2)A_P(3,4,\ldots,b-3),
\]
we have
\[
A_P(3,4,\ldots,b-2)\equiv 1 \pmod{b-2}.
\]
But in this exceptional case
\[
\frac{b!}{\operatorname{lcm}(b-1,b)}
=
(b-2)!,
\]
which is divisible by \(b-2\), a contradiction.

Thus the only compatible pair is \((a,b)=(2,3)\).
\end{proof}

By Lemma~\ref{lemma: cyclic dilation}, it immediately follows that $\beta_P(j,t)\geq\alpha_P(j,t)$ for every $t\geq j+1$.

It remains to prove the quantitative lower bound. The set $\mathcal{D}_P(j,j+2)$ consists exactly of the following elements
\[
d_2=(3;4,5,\dots,j+2),
\]
and, for $3\le a\le j+2$,
\[
d_a=(2;3,4,\dots,\widehat a,\dots,j+2).
\]
For every $2\le a\le j+2$,
\[
m(d_a)=\frac{(j+2)!}{a}.
\]
Thus the aligned classes for $\alpha_P(j,j+2)$ are
\[
x\equiv 0 \pmod{(j+2)!/a},\qquad 2\le a\le j+2.
\]
Noting that $M_P(j, j+2) = (j+2)!$, we see that inside $\mathbb Z/(j+2)!\mathbb Z$, this is the union of the unique subgroups of orders $2,3,\dots,j+2$. Hence its size is the number of elements of additive order at most $j+2$, namely
\(
\sum_{r=1}^{j+2}\varphi(r).
\)
Therefore
\begin{align}\label{eqn: alpha j j+2}
    \alpha_{P}(j,j+2) = \frac{1}{(j+2)!}\sum_{r=1}^{j+2}\varphi(r).
\end{align}
For the unit-shifted classes, write
\[
C_a:=\{x\in \mathbb Z/(j+2)!\mathbb Z:x\equiv A_P(d_a)\pmod{(j+2)!/a}\}.
\]
Each $C_a$ has size $a$. Moreover $C_2$ and $C_3$ intersect in exactly one residue class, because $A_P(d_2) = A_P(d_3)$. All other pairs are disjoint as 
\[
C_a\cap C_b\neq\varnothing
\iff
\gcd((j+2)!/a,(j+2)!/b)\mid A_P(d_a)-A_P(d_b),
\]
and the above fails except for the pair $\{2,3\}$ by Lemma~\ref{lemma: CRT incompatibility}.

So
\[
|C_2\cup C_3\cup\cdots\cup C_{j+2}| = \sum_{a=2}^{j+2}a-1 = \frac{(j+2)(j+3)}{2}-2,
\]
and therefore
\[
\beta_{P}(j,j+2) = \frac{\frac{(j+2)(j+3)}{2}-2}{(j+2)!}.
\]
Thus
\[
\beta_{P}(j,j+2)-\alpha_{P}(j,j+2) = \frac{\frac{(j+2)(j+3)}{2}-2-\sum_{r=1}^{j+2}\varphi(r)}{(j+2)!}.
\]

This is already strictly positive. Indeed,
\[
\sum_{r=1}^{j+2}\varphi(r) \le 1+\sum_{r=2}^{j+2}(r-1) = 1+\frac{(j+1)(j+2)}{2},
\]
so the numerator is at least
\[
\frac{(j+2)(j+3)}{2}-2-\left(1+\frac{(j+1)(j+2)}{2}\right) = j-1>0.
\]
Hence we have
\[
\beta_{P}(j,j+2)-\alpha_{P}(j,j+2) \ge \frac{j-1}{(j+2)!} > 0.
\]

\subsection{Proof of Proposition~\ref{propn: limit non-existence engel}}

Recall, for \(j\geq 2\) and \(q_2,\ldots,q_j\in\mathbb Z\),
\[
A_E(q_2,\ldots,q_j)
:=
-\left(
1+q_j+q_{j-1}q_j+\cdots+q_2q_3\cdots q_j
\right).
\]
As mentioned in \eqref{eq:convention-A}, we will adopt the convention
\[
A_E(\varnothing):=-1.
\]
Thus, if \(\mathbf q=(q_2,\ldots,q_j)\) and \(q\) is appended at the end, then
\begin{align}\label{eqn: recurrence for AE}
A_E(\mathbf q,q)=qA_E(\mathbf q)-1.
\end{align}

\begin{lemma}\label{lemma: engel tail stripping}
Let \(\mathbf u,\mathbf v\) be two finite integer tuples, and let
\(\mathbf t=(t_1,\ldots,t_r)\) be a common tail. Then
\[
A_E(\mathbf u,\mathbf t)-A_E(\mathbf v,\mathbf t)
=
\left(\prod_{\ell=1}^r t_\ell\right)
\bigl(A_E(\mathbf u)-A_E(\mathbf v)\bigr).
\]
\end{lemma}

\begin{proof}
Using the recurrence~\eqref{eqn: recurrence for AE}, we get
\[
A_E(\mathbf u,q)-A_E(\mathbf v,q)
=
\bigl(qA_E(\mathbf u)-1\bigr)
-
\bigl(qA_E(\mathbf v)-1\bigr)
=
q\bigl(A_E(\mathbf u)-A_E(\mathbf v)\bigr).
\]
Iterating this identity over \(t_1,\ldots,t_r\) gives the stated formula.
\end{proof}

\begin{lemma}\label{lemma: engel three-band data}
Let \(j\geq2\). For \(0\leq r\leq j-1\), define
\[
d_r
:=
\left(
2;
\underbrace{2,\ldots,2}_{j-1-r},
\underbrace{3,\ldots,3}_{r}
\right),
\]
and define
\[
d_*:=(3;3,\ldots,3).
\]
Then
\[
\mathcal D_E(j,3)
=
\{d_0,d_1,\ldots,d_{j-1},d_*\}.
\]
Moreover,
\[
m(d_r)=2^{j-r}3^r,
\qquad
m(d_*)=3^j,
\]
and
\[
A_E(d_r)
\equiv
\frac{3^r+1}{2}
\pmod{2^{j-r}3^r}.
\]
\end{lemma}

\begin{proof}
Since the quotients in the Engel case satisfy
\[
p\leq q_2\leq\cdots\leq q_j\leq3,
\]
the first component is either \(2\) or \(3\). If the first component is \(3\),
then every \(q_i\) is equal to \(3\), which gives \(d_*\). If the first
component is \(2\), then the nondecreasing tuple
\((q_2,\ldots,q_j)\) consists of some number of \(2\)'s followed by some
number of \(3\)'s, which gives \(d_r\) for a unique
\(0\leq r\leq j-1\).

The formulas for the moduli are immediate. To calculate the shift, put
\[
s:=j-1-r.
\]
The recurrence~\eqref{eqn: recurrence for AE} gives
\[
A_E(\underbrace{2,\ldots,2}_{s})
=
1-2^{s+1}
\]
and
\[
A_E(\underbrace{3,\ldots,3}_{r})
=
-\frac{3^{r+1}-1}{2}.
\]
Applying Lemma~\ref{lemma: engel tail stripping} with the common tail
consisting of \(r\) copies of \(3\), we obtain
\[
\begin{aligned}
A_E(
\underbrace{2,\ldots,2}_{s},
\underbrace{3,\ldots,3}_{r})
&=
A_E(\underbrace{3,\ldots,3}_{r})
+
3^r\bigl(
A_E(\underbrace{2,\ldots,2}_{s})-A_E(\varnothing)
\bigr)\\
&=
-\frac{3^{r+1}-1}{2}
+
3^r\bigl(2-2^{s+1}\bigr)\\
&=
\frac{3^r+1}{2}-2^{s+1}3^r.
\end{aligned}
\]
Since
\[
2^{s+1}3^r=2^{j-r}3^r=m(d_r),
\]
this gives
\[
A_E(d_r)
\equiv
\frac{3^r+1}{2}
\pmod{m(d_r)}.
\]
\end{proof}

By Lemma~\ref{lemma: cyclic dilation}, it immediately follows that
\[
\beta_E(j,t)\geq\alpha_E(j,t)
\]
for every \(t\geq2\). To prove that the inequality is strict for
\(j\geq2\), we consider the band \(t=3\).

By Lemma~\ref{lemma: engel three-band data}, the common modulus of the
congruence classes associated with \(\mathcal D_E(j,3)\) is
\[
\operatorname{lcm}
\left(
2^j,2^{j-1}3,\ldots,2\cdot3^{j-1},3^j
\right)
=
6^j.
\]

The aligned classes for \(\alpha_E(j,3)\) are
\[
x\equiv0\pmod{2^{j-r}3^r},
\qquad
0\leq r\leq j-1,
\]
together with
\[
x\equiv0\pmod{3^j}.
\]
Suppose first that \(x\not\equiv0\pmod{3^j}\), and let \(b\) be the unique
integer with \(0\leq b\leq j-1\) such that
\[
3^b\mid x,
\qquad
3^{b+1}\nmid x.
\]
Then \(x\) belongs to the aligned union if and only if
\[
2^{j-b}\mid x.
\]
Indeed, the divisibility
\[
2^{j-r}3^r\mid x
\]
can hold only when \(r\leq b\), and among such \(r\) the weakest condition
on the \(2\)-part occurs when \(r=b\).

By the Chinese remainder theorem, the \(2\)-primary and \(3\)-primary
conditions are independent. The proportion of residue classes modulo
\(3^j\) having exact \(3\)-adic valuation \(b\) is
\[
\frac{2}{3^{b+1}},
\]
while the proportion divisible by \(2^{j-b}\) modulo \(2^j\) is
\[
\frac{1}{2^{j-b}}.
\]
The class \(x\equiv0\pmod{3^j}\) has density \(3^{-j}\). Therefore
\begin{align}\label{eqn: alpha j 3}
\begin{aligned}
\alpha_E(j,3)
&=
\frac{1}{3^j}
+
\sum_{b=0}^{j-1}
\frac{2}{3^{b+1}}\frac{1}{2^{j-b}}\\
&=
\frac{1}{3^j}
+
\frac{2}{3\cdot2^j}
\sum_{b=0}^{j-1}\left(\frac{2}{3}\right)^b\\
&=
\frac{1}{3^j}
+
\frac{2}{2^j}
\left(1-\left(\frac{2}{3}\right)^j\right)\\
&=
2^{1-j}-3^{-j}.
\end{aligned}
\end{align}

We next consider the unit-shifted classes. By
Lemma~\ref{lemma: engel three-band data}, the first three such classes are
\[
C_0
:=
\left\{
x\in\mathbb Z/6^j\mathbb Z:
x\equiv1\pmod{2^j}
\right\},
\]
\[
C_1
:=
\left\{
x\in\mathbb Z/6^j\mathbb Z:
x\equiv2\pmod{3\cdot2^{j-1}}
\right\},
\]
and, when \(j\geq3\),
\[
C_2
:=
\left\{
x\in\mathbb Z/6^j\mathbb Z:
x\equiv5\pmod{9\cdot2^{j-2}}
\right\}.
\]

The classes \(C_0\) and \(C_1\) are disjoint. Indeed, by the Chinese
remainder theorem they intersect if and only if
\[
\gcd(2^j,3\cdot2^{j-1})=2^{j-1}
\]
divides \(2-1=1\), which is impossible for \(j\geq2\).

For \(j\geq3\), the classes \(C_1\) and \(C_2\) are also disjoint, since
\[
\gcd(3\cdot2^{j-1},9\cdot2^{j-2})
=
3\cdot2^{j-2}
\]
does not divide \(5-2=3\).

Finally,
\[
C_0\cap C_2\neq\varnothing
\]
if and only if
\[
\gcd(2^j,9\cdot2^{j-2})=2^{j-2}
\]
divides \(5-1=4\). Thus \(C_0\) and \(C_2\) intersect precisely when
\(j=3\) or \(j=4\). In either of these cases, their intersection is a
single residue class modulo
\[
\operatorname{lcm}(2^j,9\cdot2^{j-2})=9\cdot2^j,
\]
and hence has density
\[
\frac{1}{9\cdot2^j}.
\]

When \(j=2\), the classes \(C_0\) and \(C_1\) are disjoint, and therefore
\[
\beta_E(2,3)
\geq
\frac{1}{2^2}+\frac{1}{3\cdot2}
=
\frac{5}{12}.
\]
Since
\[
\alpha_E(2,3)
=
\frac12-\frac19
=
\frac{7}{18},
\]
we obtain
\[
\beta_E(2,3)-\alpha_E(2,3)
\geq
\frac{1}{36}.
\]

When \(j=3\) or \(j=4\), inclusion-exclusion for
\(C_0,C_1,C_2\) gives
\[
\begin{aligned}
\beta_E(j,3)
&\geq
\frac{1}{2^j}
+
\frac{1}{3\cdot2^{j-1}}
+
\frac{1}{9\cdot2^{j-2}}
-
\frac{1}{9\cdot2^j}\\
&=
\frac{2}{2^j}
=
2^{1-j}.
\end{aligned}
\]
Consequently,
\[
\beta_E(j,3)-\alpha_E(j,3)
\geq
\frac{1}{3^j}.
\]

When \(j\geq5\), the classes \(C_0,C_1,C_2\) are pairwise disjoint, and
hence
\[
\begin{aligned}
\beta_E(j,3)
&\geq
\frac{1}{2^j}
+
\frac{1}{3\cdot2^{j-1}}
+
\frac{1}{9\cdot2^{j-2}}\\
&=
\frac{19}{9\cdot2^j}.
\end{aligned}
\]
It follows that
\[
\beta_E(j,3)-\alpha_E(j,3)
\geq
\frac{1}{9\cdot2^j}+\frac{1}{3^j}.
\]

Combining the three cases, we obtain, for every \(j\geq2\),
\[
\beta_E(j,3)-\alpha_E(j,3)
\geq
\frac{1}{4\cdot3^j}>0.
\]

\subsection{Proof of Proposition~\ref{propn: bounds on limsup and liminf}}
\begin{lemma}\label{lem:monotonicity of alpha}
For fixed \(j\geq2\), the sequences
\(\{\alpha_P(j,t)\}_{t\geq j+1}\) and \(\{\alpha_E(j,t)\}_{t\geq2}
\)
are nondecreasing.
\end{lemma}
\begin{proof}
Fix \(\sigma\in\{P,E\}\), and write
\[
M_t:=M_\sigma(j,t),
\qquad
M_{t+1}:=M_\sigma(j,t+1).
\]
Since
\(
\mathcal D_\sigma(j,t)\subseteq\mathcal D_\sigma(j,t+1),
\)
we have
\[
M_t\mid M_{t+1}.
\]

For any positive integer \(M\) divisible by \(m(d)\) for every
\(d\in\mathcal D_\sigma(j,t)\), put
\[
V_{j,t}(M)
:=
\bigcup_{d\in\mathcal D_\sigma(j,t)}
\left\{
x\in\mathbb Z/M\mathbb Z:
x\equiv0\pmod{m(d)}
\right\}.
\]
By definition,
\[
\alpha_\sigma(j,t)=\frac{|V_{j,t}(M_t)|}{M_t}.
\]

Consider the natural reduction map
\[
\pi:\mathbb Z/M_{t+1}\mathbb Z
\longrightarrow
\mathbb Z/M_t\mathbb Z.
\]
Every residue class modulo \(M_t\) has exactly \(M_{t+1}/M_t\)
preimages under \(\pi\). We claim that
\[
\pi^{-1}\bigl(V_{j,t}(M_t)\bigr)
=
V_{j,t}(M_{t+1}).
\]
Indeed, for \(x\in\mathbb Z/M_{t+1}\mathbb Z\), its reduction
\(\pi(x)\) belongs to \(V_{j,t}(M_t)\) if and only if
\[
\pi(x)\equiv0\pmod{m(d)}
\]
for some \(d\in\mathcal D_\sigma(j,t)\). Since
\[
m(d)\mid M_t\mid M_{t+1},
\]
this is equivalent to
\[
x\equiv0\pmod{m(d)},
\]
which is precisely the condition that
\(x\in V_{j,t}(M_{t+1})\).

It follows that
\begin{align*}
|V_{j,t}(M_{t+1})|
=
\left|\pi^{-1}\bigl(V_{j,t}(M_t)\bigr)\right|=
\frac{M_{t+1}}{M_t}|V_{j,t}(M_t)|=
\alpha_\sigma(j,t)M_{t+1}.
\end{align*}
Finally, the inclusion
\[
\mathcal D_\sigma(j,t)\subseteq\mathcal D_\sigma(j,t+1)
\]
implies
\[
V_{j,t}(M_{t+1})
\subseteq
V_{j,t+1}(M_{t+1}).
\]
Therefore
\[
\alpha_\sigma(j,t)
=
\frac{|V_{j,t}(M_{t+1})|}{M_{t+1}}
\leq
\frac{|V_{j,t+1}(M_{t+1})|}{M_{t+1}}
=
\alpha_\sigma(j,t+1),
\]
which proves the claimed monotonicity.
\end{proof}

\begin{lemma}\label{lem:improved lower bound for liminf}
For every fixed integer \(j\geq 2\), one has
\[
\liminf_{n\to\infty}\frac{s_j(n)}{n}
\geq
\frac{1+\displaystyle\sum_{r=1}^{j+2}\varphi(r)}{(j+3)!}
>
\frac{1}{(j+2)!}.
\]
\end{lemma}
\begin{proof}
By Proposition~\ref{propn: liminf expression pierce} and summation by
parts,
\[
\liminf_{n\to\infty}\frac{s_j(n)}{n}
=
\frac{\alpha_P(j,j+1)}{j+2}
+
\sum_{t=j+2}^{\infty}
\frac{\alpha_P(j,t)-\alpha_P(j,t-1)}{t+1}.
\]
Lemma~\ref{lem:monotonicity of alpha} shows that every term in the sum
is nonnegative. Hence
\[
\liminf_{n\to\infty}\frac{s_j(n)}{n}
\geq
\frac{\alpha_P(j,j+1)}{j+2}
+
\frac{\alpha_P(j,j+2)-\alpha_P(j,j+1)}{j+3}.
\]
The set \(\mathcal{D}_P(j,j+1)\) consists of the single tuple
\[
d=(2;3,4,\ldots,j+1).
\]
For this tuple,
\[
m(d)=2\cdot 3\cdots(j+1)=(j+1)!.
\]
Hence
\begin{align*}
\alpha_P(j,j+1)=\frac{1}{(j+1)!}
\end{align*}
and, by~\eqref{eqn: alpha j j+2},
\[
\alpha_P(j,j+2)
=
\frac{1}{(j+2)!}\sum_{r=1}^{j+2}\varphi(r).
\]
Therefore
\[
\liminf_{n\to\infty}\frac{s_j(n)}{n}
\geq
\frac{1+\displaystyle\sum_{r=1}^{j+2}\varphi(r)}{(j+3)!}.
\]
Finally,
\[
\sum_{r=1}^{j+2}\varphi(r)-(j+2)
\geq
(\varphi(3)-1)+(\varphi(4)-1)>0,
\]
which proves that the last expression is greater than \(1/(j+2)!\).
\end{proof}

\begin{lemma}
For fixed $j\geq 2$, one has
\[
\limsup_{n\to\infty}\frac{s_j(n)}{n}
\leq
\exp\left(-\frac{j}{e}+O(\log j)\right).
\]
\end{lemma}

\begin{proof}
By Proposition~\ref{propn: limsup expression pierce},
\[
\limsup_{n\to\infty}\frac{s_j(n)}{n}
=
\sum_{t\geq j+1}\frac{\beta_P(j,t)}{(t+1)(t+2)}.
\]
Define
\[
\Sigma_{j,t}:=\sum_{d\in \mathcal{D}_P(j,t)}\frac{1}{m(d)}.
\]
We have
\begin{align}\label{eqn: beta bound}
\beta_P(j,t)\leq \Sigma_{j,t}.
\end{align}
Moreover,
\[
\Sigma_{j,t}
=
\sum_{\substack{p<q_2<\cdots<q_j\leq t\\ p\ \mathrm{prime}}}
\frac{1}{p q_2\cdots q_j}.
\]
Bounding the ordered sum over $q_2,\ldots,q_j$ by an unordered product gives
\[
\Sigma_{j,t}
\leq
\frac{1}{(j-1)!}
\left(\sum_{p\leq t}\frac{1}{p}\right)
\left(\sum_{n\leq t}\frac{1}{n}\right)^{j-1}.
\]
Using Mertens' theorem and the standard bound on the harmonic sum
\[
\sum_{p\leq t}\frac{1}{p}\ll \log\log(et),
\qquad
\sum_{n\leq t}\frac{1}{n}\le 1+\log t,
\]
we obtain
\begin{align}\label{eqn: Sjt bound}
\Sigma_{j,t}
\ll
\frac{\log\log(et)(1+\log t)^{j-1}}{(j-1)!}.
\end{align}

Since $(t+1)(t+2)\asymp t^2$, we see that
\[
\limsup_{n\to\infty}\frac{s_j(n)}{n}\ll
\sum_{t\geq j+1}
\frac{\beta_P(j, t)}{t^2}.
\]
Let
\[
T:=\exp((j-1)/e).
\]
We split the sum into the ranges $t\leq T$ and $t>T$.

For the tail range $t>T$, we use the bound $\beta_P(j, t)\leq 1$ that one gets from~\eqref{eqn: beta alternate defn}. Thus the tail is bounded above by
\[
\sum_{t>T}\frac{1}{t^2}
\ll
\frac{1}{T}
=
\exp(-(j-1)/e).
\]

It remains to estimate the initial range $t\leq T$. Using~\eqref{eqn: beta bound} and~\eqref{eqn: Sjt bound} we get
\[
\sum_{t\leq T}
\frac{\beta_P(j, t)}{t^2}
\ll
\frac{1}{(j-1)!}
\sum_{t\leq T}
\frac{\log\log(et)(1+\log t)^{j-1}}{t^2}.
\]
Estimating the sum by an integral and making the change of variables
\[
x=1+\log u,
\qquad
\frac{du}{u^2}=e^{1-x}\,dx,
\]
we get
\[
\sum_{t\leq T}
\frac{\log\log(et)(1+\log t)^{j-1}}{t^2}
\ll
\int_1^{1+\log T} e^{-x}x^{j-1}\log x\,dx.
\]
Since $\log T=(j-1)/e$, set
\[
X:=1+\frac{j-1}{e}.
\]
On the interval $1\leq x\leq X$, the function $e^{-x}x^{j-1}$ is increasing, because
\[
\frac{d}{dx}\log(e^{-x}x^{j-1})
=
-1+\frac{j-1}{x}>0
\]
for $x\leq X<j-1$, once $j-1$ is large. Hence
\[
\int_1^X e^{-x}x^{j-1}\log x\,dx
\ll
X\log X\cdot e^{-X}X^{j-1}.
\]
Thus the contribution from $t\leq T$ is
\[
\ll
\frac{X\log X}{(j-1)!}e^{-X}X^{j-1}.
\]

By Stirling's formula,
\[
\log((j-1)!)=(j-1)\log (j-1)-(j-1)+O(\log (j-1)).
\]
Since
\[
X=1+\frac{j-1}{e},
\]
we have
\[
\log X=\log {(j-1)}-1+O(1/(j-1)).
\]
Therefore
\begin{align*}
\log\left(\frac{e^{-X}X^{j-1}}{(j-1)!}\right)
&=
-X+(j-1)\log X-\log((j-1)!)\\
&=-\frac{j-1}{e}+O(\log {(j-1)}).
\end{align*}
The extra factor $X\log X$ contributes only $O(\log {(j-1)})$ to the logarithm. Therefore the initial range contributes at most
\[
\exp\left(-\frac{j-1}{e}+O(\log {(j-1)})\right).
\]

Combining the initial range and the tail range, we obtain the result.
\end{proof}
\subsection{Proof of Proposition~\ref{propn: bounds on engel limsup and liminf}}
By Proposition~\ref{propn: liminf expression engel} and summation by
parts,
\[
\liminf_{n\to\infty}\frac{t_j(n)}{n}
=
\alpha_E(j,2)
+
\sum_{t=2}^{\infty}
\frac{\alpha_E(j,t+1)-\alpha_E(j,t)}{t}.
\]
By Lemma~\ref{lem:monotonicity of alpha}, the summands are nonnegative,
and hence
\[
\liminf_{n\to\infty}\frac{t_j(n)}{n}
\geq
\frac{\alpha_E(j,2)+\alpha_E(j,3)}{2}.
\]
The set
\(\mathcal D_E(j,2)\) consists of the single tuple
\[
d=(2;2,\ldots,2).
\]
Its modulus is
\[
m(d)=2^j.
\]
Consequently,
\begin{align*}
\alpha_E(j,2)=\frac{1}{2^j}.
\end{align*}
Recall from~\eqref{eqn: alpha j 3} that
\begin{align*}
\alpha_E(j,3)
&=
2^{1-j}-3^{-j}.
\end{align*}
Therefore we obtain
\[
\liminf_{n\to\infty}\frac{t_j(n)}{n}
\geq
\frac{3}{2^{j+1}}-\frac{1}{2\cdot3^j}.
\]

We next prove the upper bound. By
Proposition~\ref{propn: limsup expression engel},
\[
\limsup_{n\to\infty}\frac{t_j(n)}{n}
=
\sum_{t=2}^{\infty}
\frac{\beta_E(j,t)}{t(t-1)}.
\]
Recall that \(\beta_E(j,t)\) is the density of the union of the
residue classes associated with the tuples in
\(\mathcal D_E(j,t)\). Since the residue class associated with
\(d\in\mathcal D_E(j,t)\) has density \(1/m(d)\), the union bound
gives
\[
\beta_E(j,t)\leq H_j(t),
\]
where
\[
H_j(t)
:=
\sum_{2\leq q_1\leq q_2\leq\cdots\leq q_j\leq t}
\frac{1}{q_1q_2\cdots q_j}=
\sum_{\substack{a_2,\ldots,a_t\geq 0\\
a_2+\cdots+a_t=j}}
\prod_{r=2}^{t}\frac{1}{r^{a_r}}.
\]

Multiplying by \(2^j\) and eliminating \(a_2\), we find that
\begin{align*}
2^jH_j(t)
&=
\sum_{\substack{a_3,\ldots,a_t\geq 0\\
a_3+\cdots+a_t\leq j}}
\prod_{r=3}^{t}\left(\frac{2}{r}\right)^{a_r} 
\leq
\prod_{r=3}^{t}
\left(1-\frac{2}{r}\right)^{-1} = \frac{t(t-1)}{2}.
\end{align*}
Consequently,
\[
\beta_E(j,t)
\leq 
\frac{t(t-1)}{2^{j+1}}.
\]
From the fact that $\beta_E(j, t)\leq 1$ that one gets from~\eqref{eqn: beta alternate defn}, it follows that
\[
\frac{\beta_E(j,t)}{t(t-1)}
\leq
\min\left\{
\frac{1}{t(t-1)},\frac{1}{2^{j+1}}
\right\}.
\]
Hence
\begin{equation*}
\limsup_{n\to\infty}\frac{t_j(n)}{n}
\leq
\sum_{t=2}^{\infty}
\min\left\{
\frac{1}{t(t-1)},\frac{1}{2^{j+1}}
\right\}.
\end{equation*}
Consequently, with \(T=\lceil2^{(j+1)/2}\rceil\),
\begin{align*}
\limsup_{n\to\infty}\frac{t_j(n)}{n}
&\leq
\sum_{t=2}^{T}\frac1{2^{j+1}}
+
\sum_{t>T}\frac1{t(t-1)} \\
&=
\frac{T-1}{2^{j+1}}+\frac1T
<
2^{(1-j)/2}.
\end{align*}

\section{Length Functions and Iterated Remainder Sets}\label{sec: Length Functions and Iterated Remainder Sets}

As stated earlier, a key motivation for studying the iterated remainder sets \(S_j(n)\) and \(T_j(n)\) is their connection to the length functions \(\mathcal{P}(n)\) and \(\mathcal{E}(n)\). 
We now give the proof of Lemma~\ref{lemma: engel length criterion}, which concerns the relation of $\mathcal{E}(n)$ to $T_j(n)$. 
\begin{proof}[Proof of Lemma~\ref{lemma: engel length criterion}]
By the definition of \(T_j(n)\), for every \(j\geq 0\) and every integer \(r\),
\[
r\in T_j(n)
\]
if and only if there exists \(a\) with \(1\leq a<n\) such that the Engel
trajectory starting from \(a\) reaches \(r\) at step \(j\). In particular,
\[
0\in T_j(n)
\quad\Longleftrightarrow\quad
\mathcal{E}(n,a)=j
\]
for some \(1\leq a<n\).

Suppose that \(\mathcal{E}(n)=\ell\). Since the maximum defining \(\mathcal{E}(n)\) is attained
for some \(1\leq a<n\), there exists \(a\) such that \(\mathcal{E}(n,a)=\ell\).
Hence \(a_\ell=0\), so \(0\in T_\ell(n)\).

Now let \(r\in T_\ell(n)\). Since $r\in T_\ell(n)$, there is some $b$ satisfying $1\leq b<n$ such that the trajectory from $b$ is defined through step $\ell$, so $E(n,b)\ge\ell$. On the other hand,
$E(n,b)\le E(n)=\ell$.

Hence $E(n,b)=\ell$, and therefore $r=b_\ell=0$. 

Conversely, suppose that \(T_\ell(n)=\{0\}\). Since \(0\in T_\ell(n)\),
there exists \(1\leq a<n\) such that \(a_\ell=0\). Therefore
\[
\mathcal{E}(n,a)=\ell,
\]
and hence \(\mathcal{E}(n)\geq\ell\).

If \(\mathcal{E}(n)>\ell\), then there exists \(1\leq b<n\) such that
\(\mathcal{E}(n,b)>\ell\). The trajectory starting from \(b\) has not reached \(0\)
at step \(\ell\), so
\[
b_\ell>0.
\]
Since \(b_\ell\in T_\ell(n)\), this contradicts \(T_\ell(n)=\{0\}\).
Thus \(\mathcal{E}(n)\leq\ell\), and therefore \(\mathcal{E}(n)=\ell\). 


\end{proof}
We now prove the lower bounds on the length functions.
\begin{proof}[Proof of Theorem~\ref{thm:lower-bounds-length-functions}]
We first prove the assertion for the Pierce-type length function. Let
\(j\geq 2\) be an integer such that
\[
(j+3)!<n.
\]
We consider the interval
\[
I_{j+1}^{\circ}(n)
=
\left(
\frac{n}{j+3},
\frac{n}{j+2}
\right)\cap\mathbb Z.
\]
Its underlying real interval has length
\[
\frac{n}{j+2}-\frac{n}{j+3}
=
\frac{n}{(j+2)(j+3)}
>
(j+1)!.
\]

For \(t=j+1\), the set \(\mathcal D_P(j,t)\) consists of the single
tuple
\[
\mathbf d=(2;3,4,\ldots,j+1).
\]
The modulus associated with this tuple is
\[
m(\mathbf d)
=
2\cdot 3\cdots(j+1)
=
(j+1)!.
\]

Any open interval of length greater than \(m(\mathbf d)\) contains an
integer belonging to each prescribed residue class modulo
\(m(\mathbf d)\). Hence there exists
\(
r\in I_{j+1}^{\circ}(n)
\)
such that
\[
r\equiv A_P(\mathbf d)n\pmod{m(\mathbf d)}.
\]
Lemma~\ref{lem:congruence-criterion pierce} now implies that
\(
r\in S_j(n).
\)
In particular, \(S_j(n)\) contains a positive element.

Suppose that \(\mathcal{P}(n)=\ell\leq j+1\). By Lemma~\ref{lemma: pierce length criterion},
\[
S_{\ell-1}(n)=\{0\}.
\]

It follows that if \(j\geq \ell\), then \(S_j(n)=\varnothing\), while if
\(j=\ell-1\), then
\(
S_j(n)=\{0\}.
\)
Both possibilities contradict the existence of the positive element
\(r\in S_j(n)\). Therefore
\[
\mathcal{P}(n)\geq j+2.
\]

Let \(K=K(n)\). For all sufficiently large \(n\), we have \(K\geq 5\),
so we may take \(j=K-3\). Since
\(K!<n\), this gives
\[\mathcal{P}(n)\geq K-1.
\]

It remains to estimate \(K\). By its definition,
\[
K!<n\leq (K+1)!.
\]
Stirling's formula gives
\[
\mathcal{P}(n)
\geq
K-1
=
(1-o(1))\frac{\log n}{\log\log n}.
\]

We next prove the Engel-type bound. Let \(j\geq 2\) and suppose that
\[
2^{j+1}<n.
\]
Consider the interval
\[
J_2^{\circ}(n)
=
\left(\frac n2,n\right)\cap\mathbb Z.
\]

For \(t=2\), the set \(\mathcal D_E(j,2)\) consists of the single
tuple
\[
\mathbf e=(2;2,\ldots,2),
\]
and its associated modulus is
\[
m(\mathbf e)=2^j.
\]
Since \(J_2^{\circ}(n)\) has length greater than \(2^j\), there exists
an integer
\(
r\in J_2^{\circ}(n)
\)
such that
\[
r\equiv A_E(\mathbf e)n\pmod{2^j}.
\]
Lemma~\ref{lem:congruence-criterion engel} therefore gives
\(
r\in T_j(n).
\)
In particular, \(T_j(n)\) contains a positive element.

Suppose that \(\mathcal{E}(n)=\ell\leq j\). By Lemma~\ref{lemma: engel length criterion},
\[
T_\ell(n)=\{0\}.
\]
It follows from the recursive definition that
if \(j>\ell\), then \(T_j(n)=\varnothing\), while if
\(j = \ell\), then \(T_j(n)=\{0\}\). This contradicts the existence of
the positive element \(r\in T_j(n)\). We conclude that
\[
\mathcal{E}(n)\geq j+1.
\]
For \(n\geq 9\), put
\[
m:=\left\lfloor\frac{\log(n-1)}{\log 2}\right\rfloor.
\]
Then \(m\geq 3\) and
\[
2^m\leq n-1<n.
\]
Taking \(j=m-1\), we obtain the result.
\end{proof}

\bibliographystyle{amsalpha}
\bibliography{references}

\end{document}